%% file: interp.tex
\documentclass[11pt]{article}
\def\techreport{XXX}

\usepackage{ifthen}

\usepackage{epsf,amsmath,amsfonts,amssymb,algorithm,algorithmic,psfrag,wasysym}
\usepackage{graphicx}
\usepackage{array}
\usepackage{bm}
\usepackage{algorithmic,algorithm,multicol}
\usepackage{pifont}
\usepackage{wrapfig}

\ifthenelse{\isundefined{\techreport}}{
}{
\usepackage{amsthm}
\usepackage{fullpage}
\newtheorem{theorem}{Theorem}[section]
\newtheorem{lemma}[theorem]{Lemma}
\newtheorem{proposition}[theorem]{Proposition}

\newtheorem{remark}[theorem]{Remark}

\newcommand\supp{\,\textnormal{supp}}

}

\newcounter{remark}

\include{arandmacros}

\title{Average Interpolation Under\\ the Maximum Angle Condition}

\author{Alexander Rand\thanks{Institute for Computational Engineering and Sciences, The University of Texas at Austin, 201 East 24th St, Stop C0200, Austin, Texas 78712 ({\tt arand@ices.utexas.edu})}}

\begin{document}

\maketitle

%DRAFT: \today

%\underline{References to incorporate:}

%From reviewer that gives classes of max angle tetrahedra: \cite{AADL2011}.

%Another survey that does the same thing: \cite{BHKK11}.

\begin{abstract}
Interpolation error estimates needed in common finite element applications using simplicial meshes typically impose restrictions on the both the smoothness of the interpolated functions and the shape of the simplices. 
While the simplest theory can be generalized to admit less smooth functions (e.g., functions in $H^1(\Omega)$ rather than $H^2(\Omega)$) and more general shapes (e.g., the maximum angle condition rather than the minimum angle condition), existing theory does not allow these extensions to be performed simultaneously.
By localizing over a well-shaped auxiliary spatial partition, error estimates are established under minimal function smoothness and mesh regularity. 
This construction is especially important in two cases: $L^p(\Omega)$ estimates for data in $W^{1,p}(\Omega)$ hold for meshes without any restrictions on simplex shape, and $W^{1,p}(\Omega)$ estimates for data in $W^{2,p}(\Omega)$ hold under a generalization of the maximum angle condition which requires $p>2$ for standard Lagrange interpolation. 
\end{abstract}

\ifthenelse{\isundefined{\techreport}}{
% SIAM stuff
\begin{keywords}
finite element method, average interpolation, shape regularity
\end{keywords}

\begin{AMS} 65N15, 65N30 \end{AMS}

\pagestyle{myheadings}
\thispagestyle{plain}
\markboth{A. RAND}{AVERAGE INTERPOLATION / MAXIMUM ANGLE CONDITION}
}{
}

%Original maximum angle condition: \cite{Sy57}
%Need to cite this somewhere even though it works in the $\infty$ norm: \cite{Kr92}
%Other possible references: \cite{GMW99,Du99,BG98}

Interpolation error estimates for the standard finite element spaces typically involve two requirements: conditions on the smoothness of the function being interpolated ({\em data regularity}) and constraints on the geometry of the elements of the mesh ({\em shape regularity}). 
Beyond the simplest interpolation error estimates (stated precisely in Proposition~\ref{pr:classicalinterp}), data or shape regularity can be substantially relaxed with some additional analysis. However the existing theory does not accept the weakest restrictions on both data and shape regularity simultaneously.  

\paragraph{Data Regularity and Average Interpolation} Lagrange interpolation involves point-wise function evaluation leading to data regularity restrictions associated with the requirements of the Sobolev embedding theorem.  Alternatively, average interpolation can be used to handle less regular data~\cite{Cl75,SZ90,AFW06,Sc08,CW08}.  
Because average interpolation `smears' local estimates over a neighborhood of simplices, error estimates rely more heavily on the geometry of the mesh, often needing meshes of bounded ply (defined in Section~\ref{ss:meshes}, see Figure~\ref{fg:badrefinement}) so global summation can be performed.  

\paragraph{Shape Regularity and Angle Conditions} For triangular meshes the classical error estimates are typically proved under a minimum angle restriction, but this condition is overly restrictive.  
Triangle geometry and interpolation error are actually related through the largest angle of the triangle \cite{Sy57,BA76,Ja76}.  
(Convergence of the finite element method can occur for meshes with arbitrarily large angles~\cite{BHKK11,HKK12}, but the known counterexamples are finite element spaces containing \emph{subspaces} associated with a shape regular mesh.) 
The maximum angle condition can be generalized to higher-dimensional simplices but existing analysis gives stronger data regularity requirements~\cite{Ja76,Kr92,Sh94}.  
The linear tetrahedral element is especially important to this discussion.  Shenk constructed a counterexample demonstrating that the expected Lagrange interpolation error estimates do not hold for a fairly innocuous family of narrow tetrahedra~\cite{Sh94}; see Fig.~\ref{fg:coplanarity}(right).
Analysis using an average interpolant eliminates this problem but again leads to new geometric restrictions on the ply of the mesh~\cite{Ac01}. 

\paragraph{Mesh Generation}

Geometric restrictions needed to prove interpolation error estimates are essentially the output requirements for a mesh generator. 
In 2D many mesh generation techniques aim to produce bounded aspect ratio triangulations (although a notable exception \cite{MPS07} only removes large angles), but aspect ratio guarantees are impossible when the input (i.e., the boundary of the domain) contains small angles~\cite{Sh02}. 
To admit arbitrary input several mesh generation strategies have been developed which allow a few poor quality triangles near small input angles~\cite{MPW03,PW05,Sh02,RW09}.
Guided by the interpolation theory the most commonly used strategy~\cite{MPW03} provides provable upper bounds on the largest angle and the ply of the mesh (as long as the input does not include high degree vertices). 
This effective approach does not extend to 3D and guaranteed mesh generation algorithms do not yield a bound on the ply~\cite{Li03,SG05,Si06,CDL07,RW08,RW09}.
Figure~\ref{fg:mesh}(left) demonstrates the high ply construction in the analogous 2D algorithm.
Failures of these algorithms are difficult to observe in practice since the mesh ply only grows logarithmically in the (inverse of the) size of the smallest angle, but they represent a theoretical disconnect between output guarantees of the mesh generator and the input requirements of numerical methods.
Beyond provably-correct mesh generation algorithms, simplified mesh generation algorithms are often practically successful, and allowing high ply meshes extends the range of inputs producing acceptable results. 
Uniform refinement of a structured grid in polar coordinates is perhaps the simplest situation yielding a high ply node; see Figure~\ref{fg:mesh}(right).
Providing the loosest set of geometric requirements in the interpolation theory gives mesh generators the most flexibility in practical applications as well as new challenges to improve existing guarantees under weaker requirements.

%{\bf Add more here. Also mention uniform/regular grids in polar coordinates.}
%2D: \cite{MPW03,PW05,Sh02,RW09}
%3D: \cite{Li03,SG05,Si06,CDL07,RW08,RW09} 

\paragraph{Outline} This paper demonstrates that there is no inherent trade-off between data regularity and shape regularity in interpolation of less regular Sobolev functions on simplicial meshes by developing a theory accepting minimal data and shape regularity requirements. 
The key insight is that error estimates cannot be localized over the meshes of high ply; rather a simpler auxiliary spatial partition must be used. 
Before describing this construction, Section~\ref{sc:prelim} contains the necessary preliminary discussion including a more formal exposition of the prior results.  
Error estimates for the average interpolant under uniform size and quality meshes are shown in Section~\ref{sc:uniform}.  
A generalization to non-uniform meshes is given in Section~\ref{sc:nonuniform} before some concluding remarks in Section~\ref{sc:final}.

\section{Preliminaries}\label{sc:prelim}
\setcounter{remark}{0}

\begin{figure}
\begin{tabular}{cm{.015\textwidth}c|cm{.015\textwidth}c}
\parbox{.2\textwidth}{\includegraphics[width=.2\textwidth]{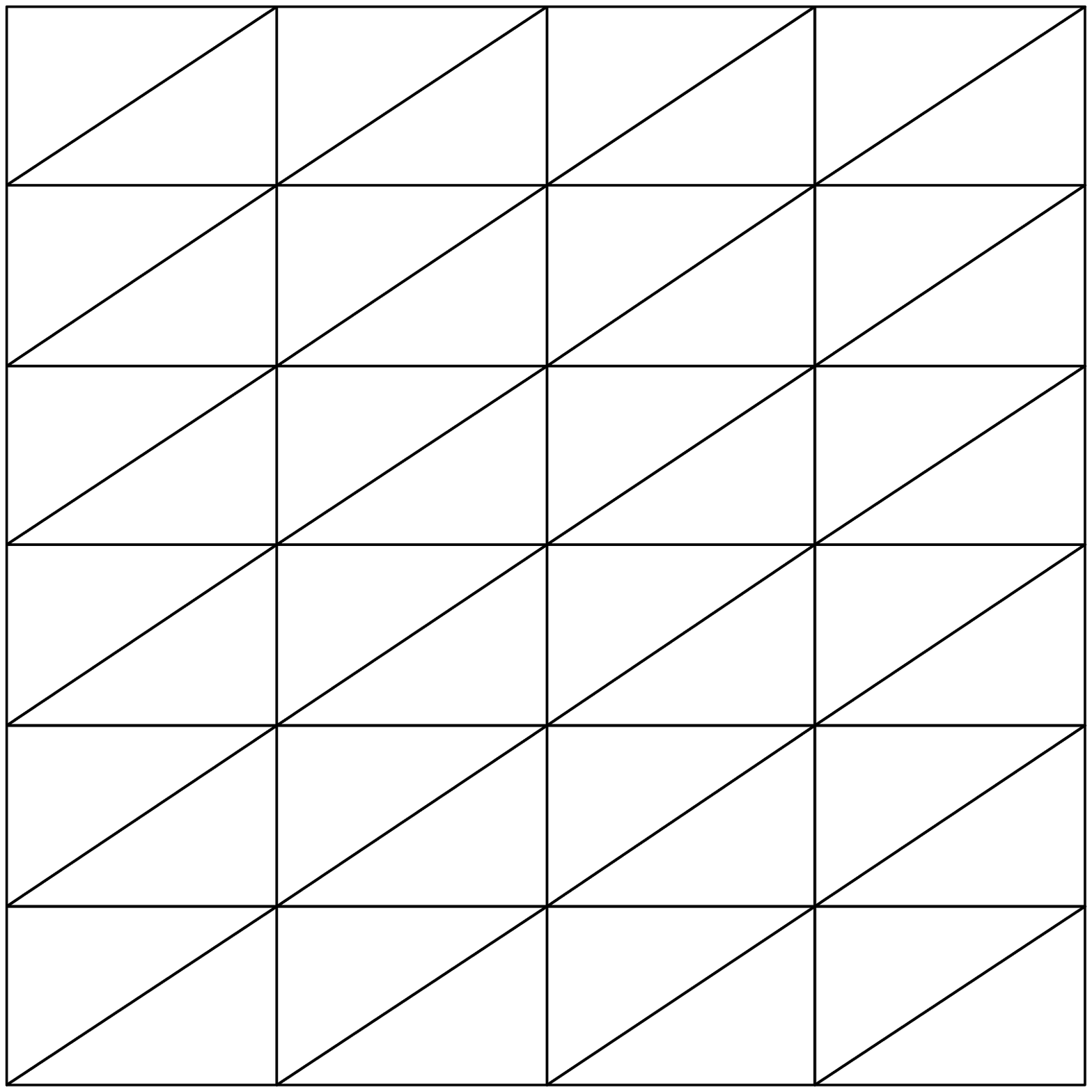}}
& \parbox{.1\textwidth}{\hspace{-.13in}{\Huge {\bf $\Rightarrow$ }}} &
\parbox{.2\textwidth}{\includegraphics[width=.2\textwidth]{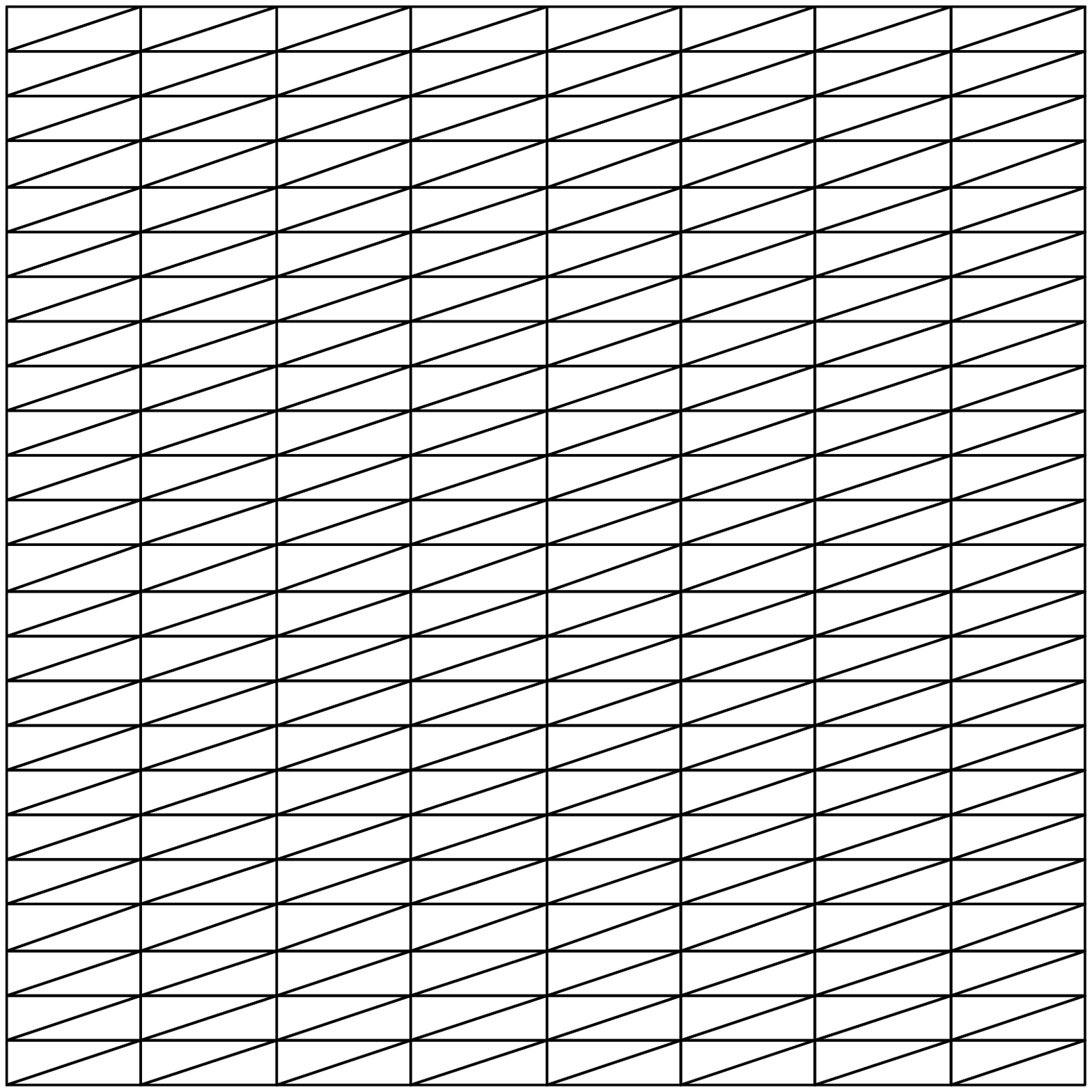}} &
\parbox{.2\textwidth}{\includegraphics[width=.2\textwidth]{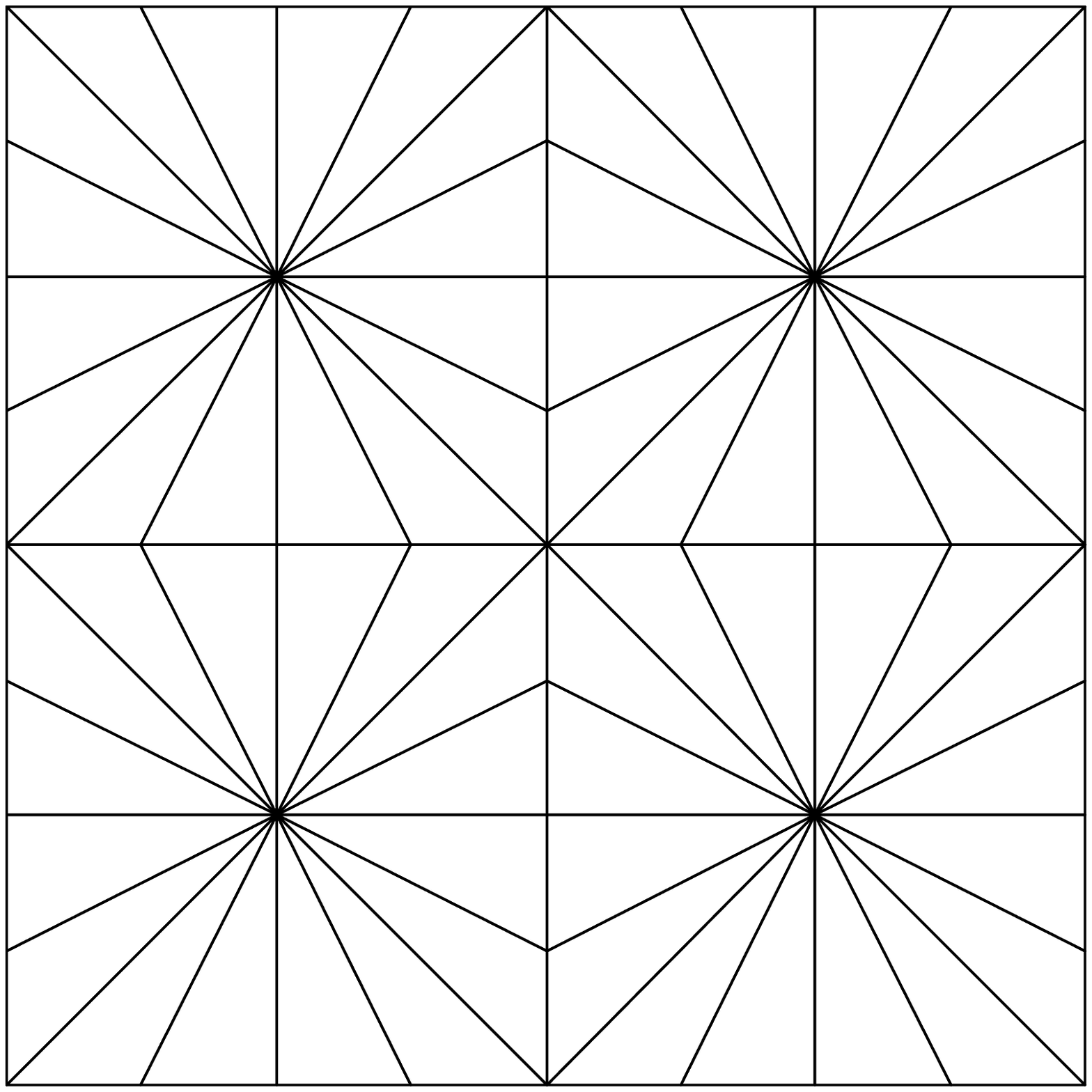}}
& \parbox{.1\textwidth}{\hspace{-.13in}{\Huge {\bf $\Rightarrow$ }}} &
\parbox{.2\textwidth}{\includegraphics[width=.2\textwidth]{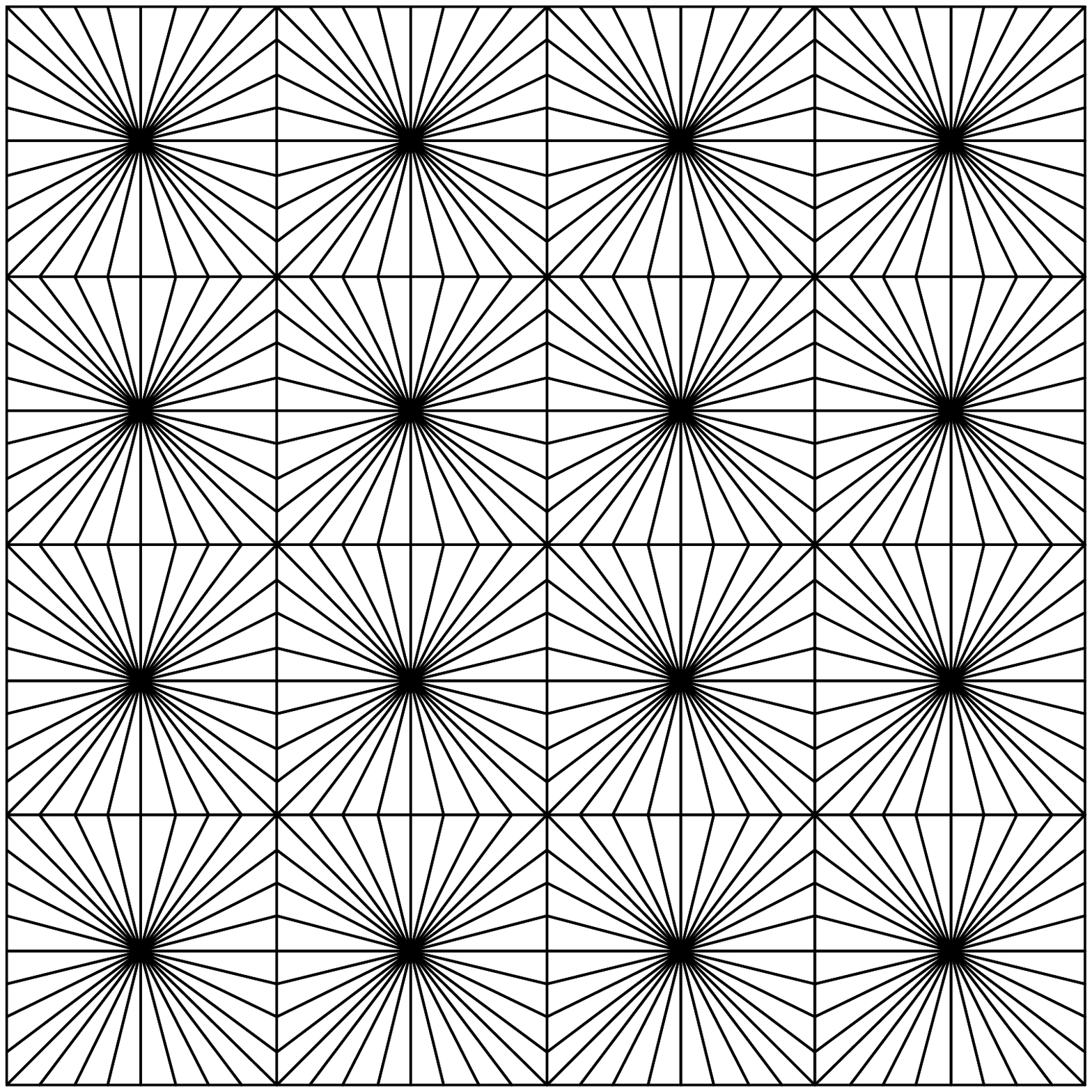}}
\end{tabular}
\caption{Two refinement strategies that do not produce bounded aspect ratio meshes. (left) A refinement strategy creating a mesh with bounded ply and maximum angle: all vertices have degree six and all angles are acute. (right) A refinement strategy creating a high ply mesh without large angles: all angles are smaller than $135$ degrees: in the limiting case, the triangles in the corners contain angles very near $135^\circ$, $45^\circ$ and $0^\circ$.}\label{fg:badrefinement}
\end{figure}

\begin{figure}
\begin{tabular}{cc}
\parbox{.43\textwidth}{
\includegraphics[width=.41\textwidth]{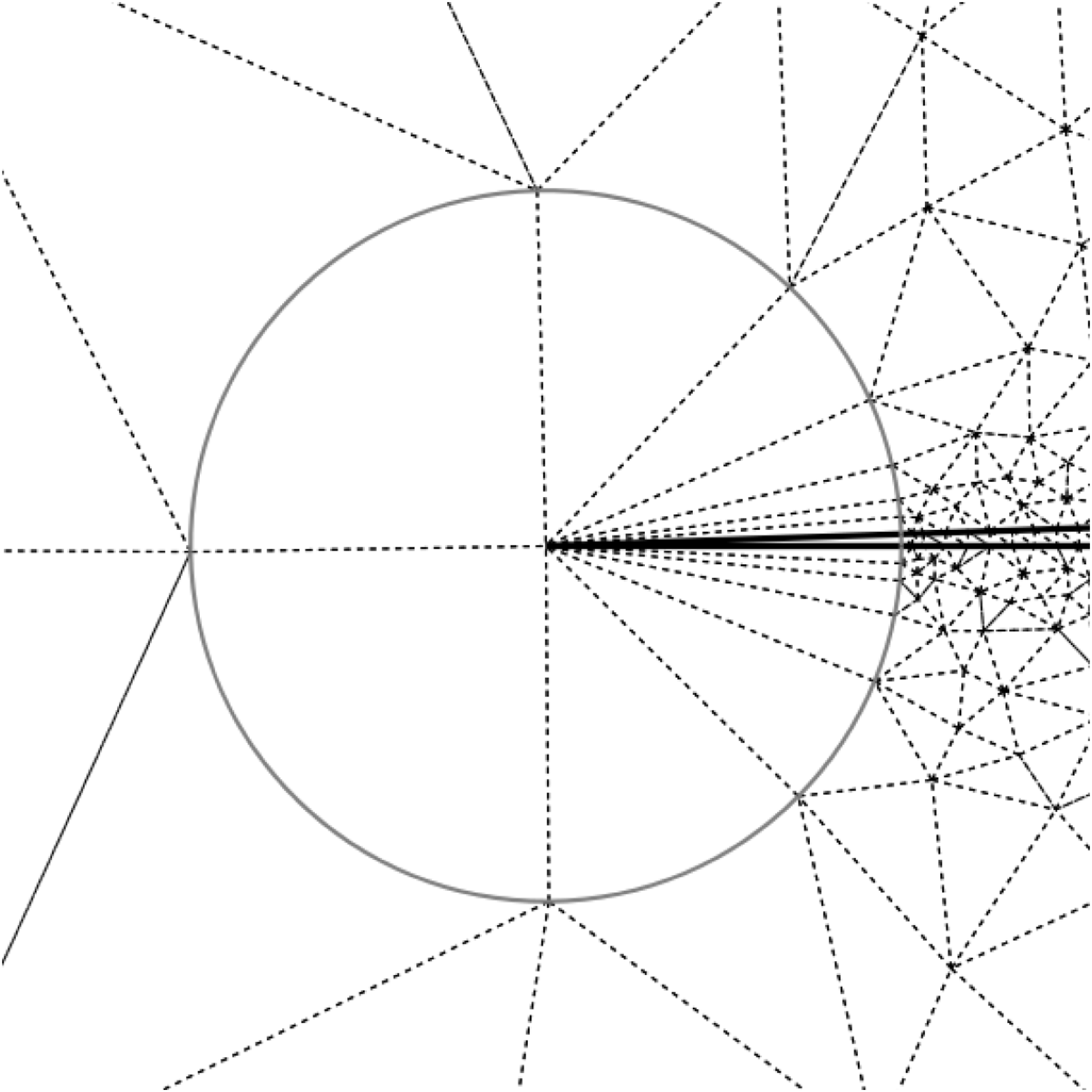}} &
\parbox{.53\textwidth}{
\includegraphics[width=.52\textwidth]{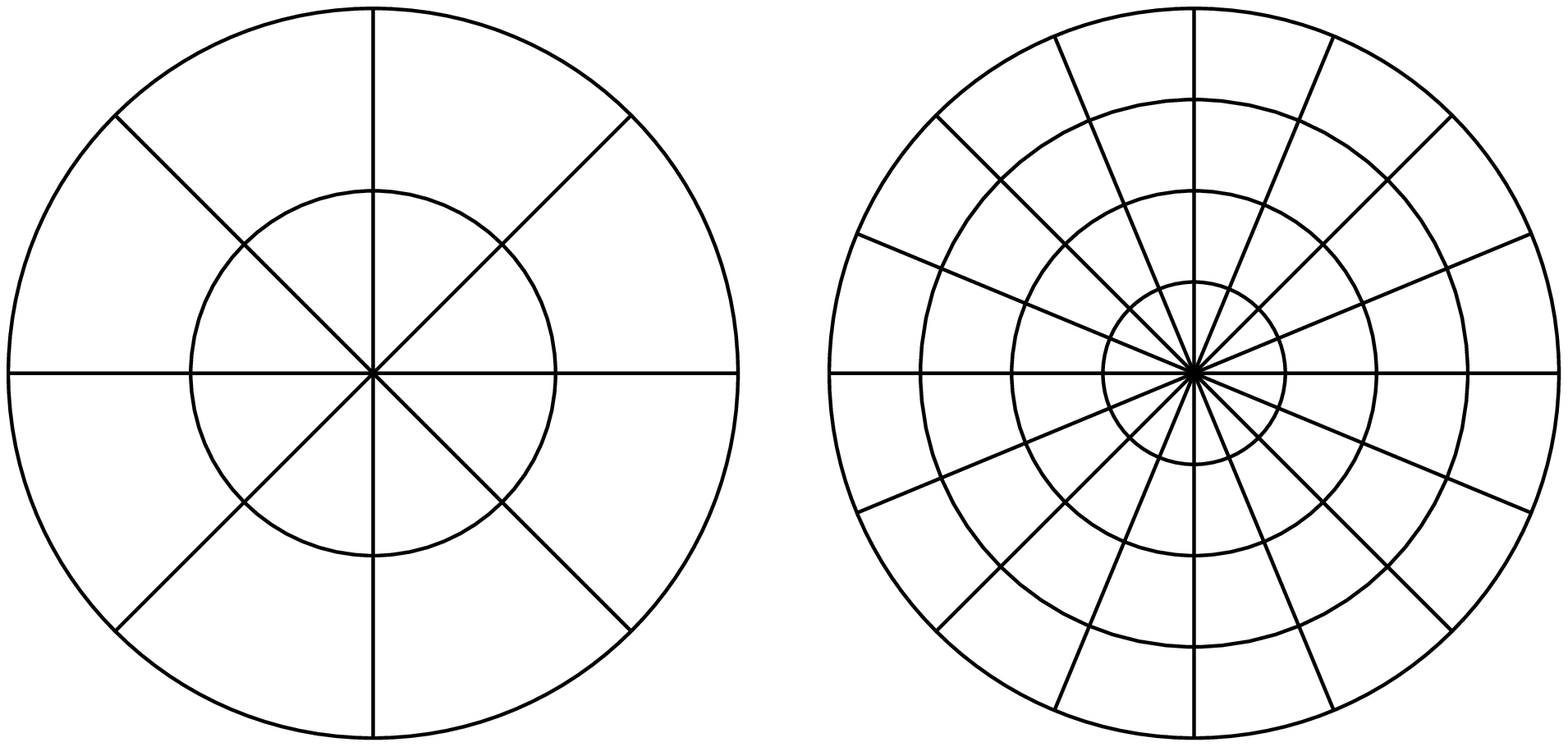}

\caption{(left) A quality mesh generated for an input with a very small angle. Certain mesh generation strategies can produce high ply meshes at vertices (or along edges in 3D) near small input angles. (above) ``Uniform'' refinement in polar coordinates leads to a high ply at the origin.}\label{fg:mesh}
}
\end{tabular}
\end{figure}

The statement  $x \apprle y$ is used to mean there exists a positive constant $C$ such that $x \leq C\, y$ when the details of the constant are unimportant.  For nearly all estimates discussed, this implied constant will be independent of simplex/mesh shape regularity; whenever this is not the case, it will be explicitly noted.  

\subsection{Balls, Cubes, and Simplices}\label{ss:basic}

The \emph{ball} of radius $r$ centered at $\bx\in \R^d$ is the set of points with Euclidean distance from $\bx$ less than $r$: 
\[
B(\bx,r) := \{ \by \, : \, \vsn{\bx-\by}<r\}.
\]
%%%% It doesn't look like I use this anywhere
%The volume of the unit ball in $d$ dimensions is denoted $\alpha_d$.  
A $d$-cube or \emph{cube} is a closed set $Q\subset \R^d$ of the form
\[
Q = \left[a_1-\frac{h}{2},a_1+\frac{h}{2}\right] \times \left[a_2-\frac{h}{2},a_2+\frac{h}{2}\right] \times \ldots \times \left[a_d-\frac{h}{2},a_d+\frac{h}{2}\right].
\]
The point $\ba = (a_1,\ldots,a_d)$ is called the \emph{center} of $Q$ and $h$ is called the \emph{size} of $Q$, denoted $\size(Q)$.  
A \emph{simplex} is the convex hull of $d+1$ non-coplanar points (called \emph{vertices}), $\{\bv_i\}_{i=0}^d$:
\[
K = \left\{ \sum_{i=0}^d t_i \bv_i \,:\, t_i > 0, \sum_{i=0}^d t_i = 1\right\}.
\]
We let $h_K$ denote the length of the longest side of $K$ and $\rho_K$ denote the radius of the largest sphere inscribed in $K$.  
Following Jamet's shape regularity condition~\cite{Ja76}, the angle $\theta_K$ is defined to characterize the quality of a simplex $K$ by the formula,
\begin{equation}\label{eq:copla}
\theta_K := \max_{\vsn{\bm{\xi}}=1} \min_{i} \, \arccos \bm{\xi} \cdot \bu_i
\end{equation}
where $\{\bu_i\}$ is the set of unit vectors parallel to the edges of the simplex $K$; see Fig.~\ref{fg:coplanarity}.  
This angle measures how far from coplanar the set of edges is, so we call $\theta_K$ the {\em coplanarity} measure.  The subscript on $h$, $\rho$, and $\theta$ is often omitted when the simplex in question is apparent.  In two dimensions $\theta_K$ is half the largest angle of the triangle, and thus the set of triangles with coplanarity bounded away from $\pi/2$ satisfy the maximum angle condition in the sense that all angles are bounded away from $\pi$. 
In three dimensions, a more intuitive generalization of the maximum angle condition defined by defined by K\v{r}\'{\i}\v{z}ek~\cite{Kr92} requires that both face and dihedral angles of the tetrahedron are bounded away from $\pi$. 
This is also equivalent to bounding coplanarity away from $\pi/2$ as shown in the proposition below.

%, which has been noted in passing in \cite{BHKK11}. We make this more precise in the proposition below.

\begin{proposition}\label{pr:jametmac}
Let $\psi_K$ denote the maximum of all face and dihedral angles of (three-dimensional) tetrahedron $K$. For tetrahedra, bounds on coplanarity and face/dihedral angles are equivalent, i.e., 
if $\theta_K$ is bounded away from $\pi/2$ then $\psi_K$ is bounded away from $\pi$ and vice versa.
\end{proposition}

The proof, given in Appendix~\ref{ap:copmac}, relies on a characterization of tetrahedra satisfying the maximum angle condition given in \cite{AADL2011}.

\begin{figure}
\begin{center}
\psfrag{ximax}{$\xi_{\rm max}$}
\psfrag{theta}{$\theta$}
\psfrag{epsilon}{$\epsilon$}
\includegraphics[width=.8\textwidth]{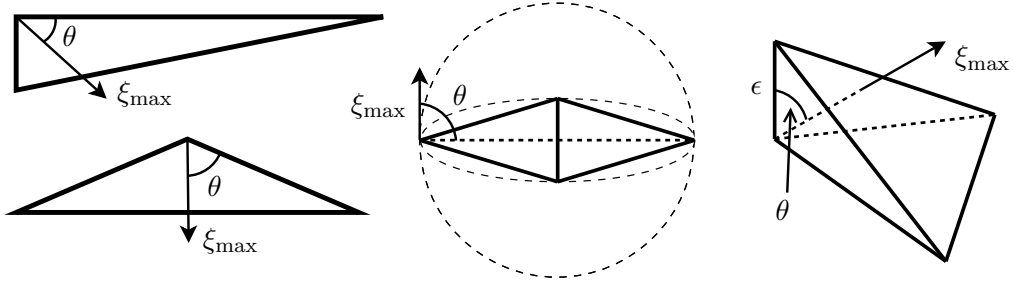}
\end{center}
\caption{Coplanarity $\theta$ for for several simplices with vector $\xi_{\rm max}$ which maximizes the expression in (\ref{eq:copla}) labeled. (left) In 2D the coplanarity is half the largest angle. (center) The coplanarity of a nearly flat ``sliver'' tetrahedron is nearly $\pi/2$. (right) The simplices used in a counter-example due to Shenk~\cite{Sh94} have coplanarity bounded away from $\pi/2$ even as the edge length $\epsilon$ approaches $0$.}\label{fg:coplanarity}
\end{figure}

\subsection{Sobolev Spaces}\label{ss:sobolev}
We will distinguish between the notation $\frac{\partial}{\partial x_i}$ and $\partial_i$. 
The former will be used to denote differentiation with respect to a particular variable while the latter represents differentiation with respect to a function argument.
For example, the following notation is used for an application of the chain rule: \[
\frac{\partial}{\partial x_i} \left[u(2\bx+3\by)\right] = 2\partial_i u(2\bx+3\by). 
\]
The Sobolev semi-norms and norms on functions over an open set $\Omega$ are defined by
\begin{align*}
\wmpsn{u}{m}{p}{\Omega}^p &:=  \int_\Omega \sum_{|\beta| = m} |\partial^\beta u(\bx)|^p \d \bx & & {\rm and} &
\wmpn{u}{m}{p}{\Omega}^p &:= \sum_{0\leq k\leq m}\wmpsn{u}{m}{p}{\Omega}^p,
\end{align*}
where the first summation is taken over multi-indices $\beta$ and differentiation with respect to multi-indices has the standard definition, $\partial^\beta u= \partial^{\beta_1} \partial^{\beta_2} \ldots u$.  
The $W^{0,p}$-norm is simply the standard $L^p$-norm.

Throughout this paper, $\Omega\subset \R^d$ is a bounded set and an extension domain for any Sobolev spaces used; i.e. there is a continuous, linear operator $E_\Omega: W^{m,p}(\Omega)\rightarrow W^{m,p}(\R^d)$ such that $u(x) = E_\Omega u(x)$ for all $x\in\Omega$; see \cite[p. 83]{Ad03} for technical results concerning extension domains.
%\[
%\wmpsn{\bar u}{m}{p}{\R^d} \apprle \wmpsn{u}{m}{p}{\Omega}.
%\]
Without ambiguity, the functions $u$ and $E_\Omega u$ will not be distinguished: for $x\notin \Omega$, $u(x)$ denotes $E_\Omega u(x)$.  Likewise, the estimate $\wmpsn{u}{m}{p}{\R^d} \apprle \wmpsn{u}{m}{p}{\Omega}$ is used freely.

\subsection{Lagrange Interpolation on Simplicial Meshes}\label{ss:meshes}
$\cT$ always denotes a simplicial mesh of $\Omega$.  
The \emph{ply} of the mesh $\cT$ is the maximum number of mesh simplices intersecting a single simplex.  Ply restrictions are often implicit in interpolation error estimates since a mesh with bounded aspect ratio has a bounded ply.  (In fact, the weaker restriction of bounded circumradius to shortest edge ratio is sufficient to ensure bounded ply~\cite{Ta97,HMP06}.)  

Let $\Pi_k$ denote the $k$-th degree Lagrange interpolant on a simplicial mesh,
\[
\Pi_k u(\bx) := \sum_i u(\bv_i) \phi_i(\bx)
\]
where $\phi$ are the piecewise nodal basis functions on $\cT$.  While this paper only considers linear interpolants ($k=1$), the higher order cases are included in the statement of classical estimates for smooth functions.

\subsection{Prior Interpolation Error Estimates}\label{ss:priorwork}
The Bramble-Hilbert Lemma \cite{BH70,DL04} along with scaling properties of Sobolev semi-norms leads to the following classical error estimate for the Lagrange interpolant; see~\cite{Ci02,EG04,BS08} for complete details.
\begin{proposition}\label{pr:classicalinterp}
If $k+1 > \frac{d}{p}$ and $0 \leq m \leq k$, then for all simplices $K$ and all $u\in W^{k+1,p}(K)$,
\begin{equation}\label{eq:classicalinterp}
\wmpsn{u - \Pi_k u}{m}{p}{K} \apprle \frac{h^{k+1}}{\rho^m}\wmpsn{u}{k+1}{p}{K}.
\end{equation}
\end{proposition}

For simplices with bounded aspect ratio, $h$ and $\rho$ are proportional and thus the factor $\frac{h^{k+1}}{\rho^m}$ reduces to $h^{k+1-m}$.  The estimate in the $L^2$-norm (i.e., the $m=0$ case) is particularly nice: the term in the denominator disappears and the estimate is independent of the shape of the simplex.  

The first generalization involves weakening the data regularity requirement, i.e., the restriction $u\in W^{2,p}(\Omega)$.
A function $u\in W^{1,p}(\Omega)$ (but not in $W^{2,p}(\Omega)$) is not necessarily continuous so an alternative to Lagrange interpolation must be considered.  Cl\'ement first described such an interpolant and proved an estimate which corresponds to the ($k=0$)-case of (\ref{eq:classicalinterp})~\cite{Cl75}.  
Scott and Zhang proved a similar estimate based on an average interpolant~\cite{SZ90}: specifically interpolant values at mesh vertices are defined to be averages over an arbitrarily selected adjacent simplex. 
A distinct advantage of the Scott and Zhang interpolant is that coefficients of basis functions corresponding to boundary nodes only depend upon the boundary data. 
An alternative approach based on averaging over balls has been developed which extends to vector-valued functions in the discrete de Rham complex~\cite{AFW06,Sc08,CW08}; see Figure~\ref{fg:acostavrand} for a comparison of the averaging regions.
The resulting interpolant of Christiansen and Winther also matches homogeneous boundary data selecting an appropriate extension of the function outside the domain.
In the simplest case, each of these constructions reduces to the following theorem.
\begin{theorem}[Specialized from \cite{SZ90,CW08}, etc.]\label{th:scottzhang}
Suppose $\cT$ has a bounded aspect ratio and let $\Pi_0$ be an average interpolant.  Then for all $u\in W^{1,p}(\Omega)$,
\begin{equation}\label{eq:scottzhang}
\lpn{u - \Psi_0 u}{p}{K} \apprle h \wmpsn{u}{1}{p}{\hat K},
\end{equation}
where $\hat K$ is the union of $K$ and its neighboring simplices. 
\end{theorem}

Unlike the ($m=0$, $k=1$)-case of Proposition~\ref{pr:classicalinterp}, Theorem~\ref{th:scottzhang} requires a bounded aspect ratio mesh ensuring two things. 
First, $K$ and its neighboring simplices have roughly the same size so that $\diam(K) \approx \diam(\hat K)$. 
Second, $K$ has a bounded ply so these local estimates (\ref{eq:scottzhang}) over each simplex can be summed to a global estimate.  
Our aim is to remove these restrictions on simplex shape to match the $m=0$ case of the classical estimate (\ref{eq:classicalinterp}) which is independent of simplex shape.

The second generalization of Proposition~\ref{pr:classicalinterp} weakens the shape regularity requirements.
For this purpose, Jamet developed an improved estimate for sufficiently regular data.
\begin{theorem}[Jamet, \cite{Ja76}]\label{th:jamet}
If $k + 1 - m > \frac{d}{p}$, then for all $u\in W^{k+1,p}(K)$,
\begin{equation}\label{eq:jamet}
\wmpsn{u - \Pi_k u}{m}{p}{K} \apprle \frac{h^{k+1 - m}}{(\cos \theta)^m}\wmpsn{u}{k+1}{p}{K}.
\end{equation}
\end{theorem}

The restriction $k + 1 - m > \frac{d}{p}$ prevents Jamet's result from applying in certain cases including (perhaps the most important cases of) linear interpolation on triangles and tetrahedra: $k = m = 1$ when $d=2$ or $d=3$.  
The analysis of Babu\v{s}ka and Aziz \cite{BA76} yields the result for the linear triangular element (although care must be taken to derive the correct geometric factor $1/\cos \theta$ using this approach \cite{Kr91,GMW99,MS09,Ra09}).  Shenk generalized the approach of Babu\v{s}ka and Aziz to higher dimensions and, more importantly, demonstrated that (\ref{eq:jamet}) does not hold for the linear Lagrange interpolant on tetrahedra~\cite{Sh94}.  Shenk's example involves particularly simple geometry: a family of tetrahedra depicted in Figure~\ref{fg:coplanarity} with vertices $(0,0,0)$, $(1,0,0)$, $(0,1,0)$, and $(0,0,\epsilon)$ where $\epsilon$ approaches zero.  Demonstrating the error estimate for an analogous class of triangles (i.e., 2D simplices) is the crux of the argument of Babu\v{s}ka and Aziz.  

An analysis by Acosta demonstrated that expected error estimates can be extended to certain narrow tetrahedra (including those in Shenk's example) by using average (rather than Lagrange) interpolation~\cite{Ac01}.  
This method does allow data and shape regularity to be weakened simultaneously, but the analysis is inherently anisotropic: locally it must be possible to affinely transform the mesh into a shape-regular mesh.  
(This is a standard requirement in an anisotropic framework; e.g., \cite{Ap99,FP01,HK09}.) 
The resulting estimates expect the mesh to have a bounded ply so that local estimates can be summed to form global error estimates.  
Note that the meshes in Figure~\ref{fg:mesh} do not satisfy the requirements of \cite{Ac01}.

We will develop interpolation error estimates based on average interpolation under weakened data and shape regularity requirements.  Emphasis will be placed on avoiding any unnecessary element shape requirements which, most notably, allows meshes with unbounded ply (as is the case for Lagrange interpolants in Theorem~\ref{th:jamet}).  The result will be an interpolation theory which extends Jamet's minimal shape regularity requirements to less smooth data including the linear tetrahedral element.  

\section{Uniform Error Estimates}\label{sc:uniform}
\setcounter{remark}{0}
To avoid the limitations associated with pointwise function evaluation, an average interpolant is defined by mollifying the data before Lagrange interpolation:
\begin{equation}\label{eq:interpolant}
\Xi_{h}u(\bx) := \Pi_1 \cM_h u(\bx).
\end{equation}
In the next subsection, the standard mollification procedure $\cM_h$ is defined and a number of needed properties are given.  

{\it Remark \arabic{section}.\addtocounter{remark}{1}\Alph{remark}}.\label{rm:acosta}
This interpolant is used in~\cite{AFW06,Sc08,CW08}, but in that context, the mollification radius was selected to be small enough so that the averaging region stays within the neighboring triangles. Since we do not assume aspect ratio bounds, we cannot use this property.
Acosta uses essentially an anisotropic variant of this construction~\cite{Ac01}; see Figure~\ref{fg:acostavrand}.
%%% This was already stated...
%The interpolants of Cl\'ement~\cite{Cl75} and Scott/Zhang~\cite{SZ90} rely on averaging over simplices rather than balls.

\begin{figure}
\begin{center}
\includegraphics[width=.98\textwidth]{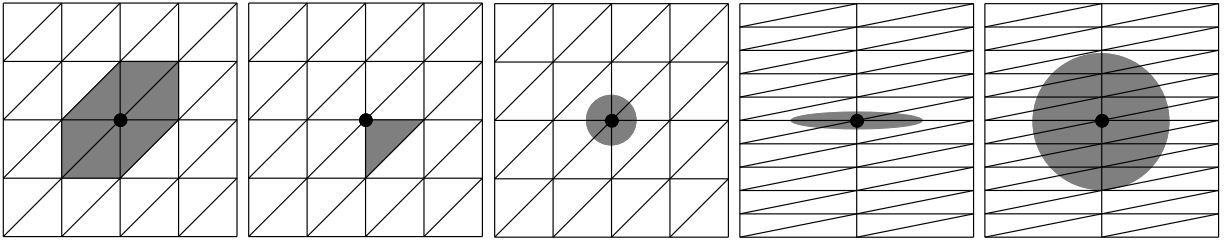}
\end{center}
\caption{
%The anisotropic averaging region used by Acosta (left) and the isotropic averaging scheme used in this paper (right).  Acosta's interpolant cannot be defined for the refinement depicted in Figure~\ref{fg:badrefinement}
Averaging regions for the interpolants of (left to right) Cl\'ement, Scott and Zhang, Christiansen and Winther, Acosta, and this paper. The first three were developed for shape regular meshes. The final two are shown on an anisotropic mesh to highlight the differences in the construction. While the Acosta interpolant requires well-defined anisotropy in the mesh, our interpolant makes minimal assumptions on the mesh quality.
}\label{fg:acostavrand}
\end{figure}

\subsection{Mollification and Integration}\label{ss:mollifiers}
Mollification is a common tool in analysis for producing smooth approximations of functions in $L^p$ spaces. After defining the procedure, a few common facts about mollification are listed. 
The proofs of these and similar results can be found in a number of textbooks, but are given in the appendix for comparison to the later sections. 
For some examples of these kinds of technical details, especially those for more regular data and including explicit dependence on on the mollification radius, see \cite[p. 191]{Na73}, \cite[p. 58]{LL01}, \cite[p. 98]{MB02}, and \cite[p. 201]{FL07}. 

Let $\rho \in C^\infty(\R^d)$ such that (i) $\supp(\rho) \subset \overline{B(0,1)}$, (ii) $\int \rho(\bx) \d \bx = 1$, and
(iii) $\rho(\bx) \geq 0$ for all $\bx$.  
Let $\rho_h (\bx) = \frac{1}{h^d} \rho(\frac{\bx}{h})$.  Then, (i) $\supp(\rho_h) \subset \overline{B(0,h)}$, (ii) $\int \rho_h(\bx) \d \bx = 1$, (iii) $\rho_h(\bx) \geq 0$ for all $\bx$, and (iv) $\vn{\nabla \rho_h}_{L^1(\R^d)} \leq \frac{1}{h}\vn{\nabla \rho}_{L^1(\R^d)}$.
The most commonly used mollifier is $\rho(\bx) = c e^{-1/(1-\vn{\bx}^2)}$ where the constant $c$ is chosen so (ii) holds.
We also require that $\rho$ is \emph{radially symmetric}. 
While symmetry is not usually required, the common mollifier is symmetric and we explicitly use the property in one step of the upcoming arguments.

A function is mollified by convolving it with $\rho_h$: 
%This is a useful technique for regularizing a non-smooth function.  
given $u\in L^1(\R^d)$, define 
$$
\cM_h u(\bx) := \int_{\R^d} u(\bx-\by) \rho_h(\by) \d \by = \int_{\R^d} u(\bx-h\bz) \rho(\bz)  \d \bz.
$$
The two representations above are equivalent, resulting from a simple change of variables, $\by=h\bz$, but the supports of the integrands are different. 
If a function is smooth enough (specifically, it belongs to $W^{1,p}(\Omega)$ in the proposition below) then the difference between the function and its mollification can be bounded in terms of the mollification radius.  
\begin{proposition}\label{pr:dermol}
If $p \geq 1$, then for all  $u\in W^{1,p}(\Omega)$,
\[\lpn{u - \cM_h u}{p}{\R^d} \apprle h \wmpsn{u}{1}{p}{\Omega}.\]
\end{proposition}

%The previous proposition will be used with respect to higher derivatives of $u$ as stated in the next corollary.  
%\begin{corollary}\label{cr:dermol}
%If $p \geq 1$, then for all  $u\in W^{m+1,p}(\Omega)$,
%[\wmpsn{u - \cM_h u}{m}{p}{\R^d} \apprle h \wmpsn{u}{m+1}{p}{\Omega}.\]
%\end{corollary}

%% don't need this anymore?
%Mollified functions are smooth and the next proposition explicitly bounds the norm of the higher derivatives of the mollified function in terms of the mollification radius.
%\begin{proposition}\label{pr:dermol2} If $p \geq 1$, then for all $u\in L^{p}(\Omega)$,
%\begin{equation}\label{eq:dermol2}
%\wmpsn{\cM_h u}{1}{p}{\R^d} \apprle \frac{1}{h} \lpn{u}{p}{\Omega}.
%\end{equation}
%\end{proposition}
%
The next proposition demonstrates how the smoothness of the mollifier depends on the mollification radius through a bound on the $L^p$-norm of a mollified function in terms of the $L^q$-norm of the original function.  
\begin{proposition}\label{pr:lpqmol}
If $1 \leq q \leq p \leq \infty$ and $Q$ is a bounded set, then for all $u\in L^q(Q)$,
\begin{equation}\label{eq:lpqmol}
\lpn{\cM_h u}{p}{Q} \apprle h^{d\frac{q-p}{pq}} \lpn{u}{q}{Q_h}
\end{equation}
where $Q_h := \bigcup_{x\in Q} B(x,h)$.  
\end{proposition}

%Since $\Omega$ is assumed to be an extension domain, this immediately implies \[\lpn{\cM_h u}{p}{\Omega} \apprle h^{d\frac{q-p}{pq}} \lpn{u}{q}{\Omega}.\]

{\it Remark \arabic{section}.\addtocounter{remark}{1}\Alph{remark}}.\label{rm:commutemol}
Since $\frac{\p}{\p x_i}\left(\cM_h u\right) = \cM_h \left( \frac{\p u}{\p x_i}\right)$, Proposition \ref{pr:lpqmol} also applies for higher Sobolev norms; specifically, (\ref{eq:lpqmol}) implies that
\begin{align}
\wmpsn{\cM_h u}{k}{p}{Q} & \apprle h^{d\frac{q-p}{pq}} \wmpsn{u}{k}{q}{Q_h}.\label{eq:wmpqmol}
\end{align}

Finally a standard estimate of the $L^p$-norm of a function in terms of the $L^q$-norm is stated for $q$ larger than $p$.  
This can be proved with either Jensen's or H\"older's inequality.
\begin{proposition}\label{pr:lpq}
If $1 \leq p \leq q \leq \infty$ and $Q$ is a bounded set, then for all $u\in L^q(Q)$, 
\[\lpn{u}{p}{Q} \leq \vsn{Q}^{\frac{q-p}{pq}} \lpn{u}{q}{Q}.\]
\end{proposition}

\subsection{Error Estimate}\label{ss:uniformw1p}

Next we turn to removing the integrability condition in Jamet's estimate (\ref{eq:jamet}).  
Again for simplicity we first consider meshes in which simplex size and quality (measured by coplanarity, not aspect ratio) are uniformly bounded and then relax these assumptions in Section~\ref{sc:nonuniform}.

\begin{theorem}\label{th:w1uniform}
Suppose $\cT$ contain no edges longer than $h$ and let $\theta$ be the maximum simplex coplanarity. Then for $k\in\{0,1\}$ and for all $u\in W^{k+1,p}(\Omega)$,
\begin{equation}\label{eq:w2p}
\wmpsn{u - \Xi_{h} u}{k}{p}{\Omega} \apprle \frac{h}{(\cos \theta)^k}\wmpsn{u}{k+1}{p}{\Omega}.
\end{equation}
\end{theorem}

\begin{figure}
\begin{center}
\includegraphics[width=.28\textwidth]{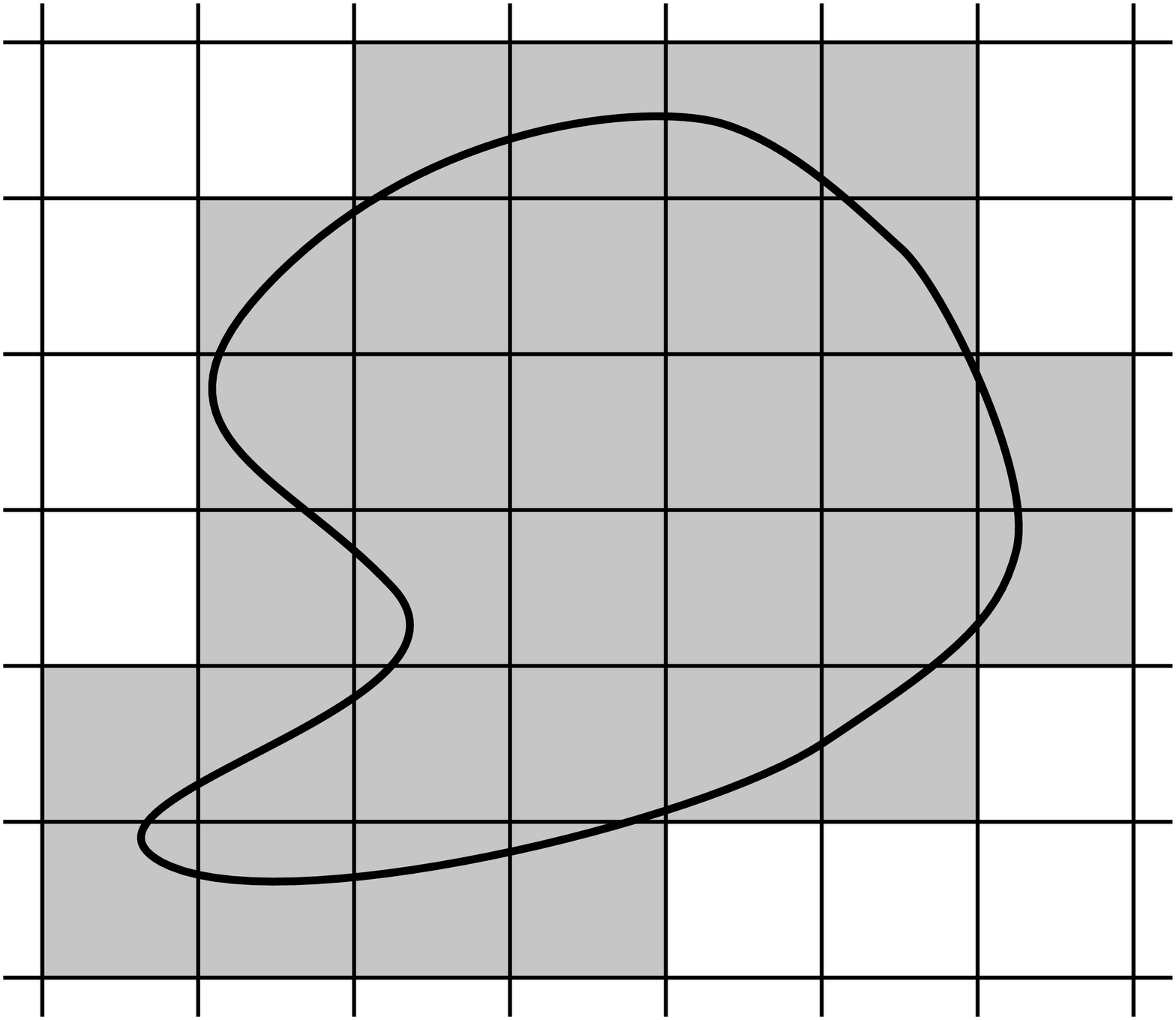}
\end{center}
\caption{For an example domain, the cubes in $\cQ$ are shaded.}\label{fg:covermesh}
\end{figure}

\begin{proof}
First the triangle inequality is applied using the intermediate function $\cM_h u$,  
\begin{equation}\label{eq:twoparts}
\wmpsn{u - \Xi_{h} u}{k}{p}{\Omega} \leq \wmpsn{u - \cM_h u}{k}{p}{\Omega} + \wmpsn{\cM_h u - \Xi_{h} u}{k}{p}{\Omega}.
\end{equation}
The first term of (\ref{eq:twoparts}) is independent of the triangulation $\cT$ and is estimated using Proposition~\ref{pr:dermol}:
\begin{equation}
\wmpsn{u - \cM_h u}{k}{p}{\Omega} \apprle h \wmpsn{u}{k+1}{p}{\Omega}.
\end{equation}
Next the second term of (\ref{eq:twoparts}) is addressed:  
\begin{equation}\label{eq:w1pkeyterm}
\wmpsn{\cM_h u - \Xi_{h} u}{k}{p}{\Omega} = \wmpsn{\cM_h u - \Pi_{1}\cM_h u}{k}{p}{\Omega}.
\end{equation} 
The estimate is established locally, but not over the individual simplices of $\cT$.  
This is important since averaging `smears' the estimate on one simplex over its neighbors.  
Since we allow for high ply meshes, it is not possible to sum these local estimates to yield the global result (\ref{eq:w2p}).  
Rather, the estimate is localized over a simple artificial (and bounded-ply) decomposition, a covering of $\Omega$ by cubes.  

In a lattice of cubes of side-length $h$, let $\cQ = \{ Q_i\}$ be the subset which intersect $\Omega$; see Figure~\ref{fg:covermesh}.  The term (\ref{eq:w1pkeyterm}) is estimated on a single cube $Q_i\in \cQ$.  
Let $R_i$ and $S_i$ denote the cubes of side-length $3h$ and $5h$, respectively, centered around $Q_i$, and let $\cK_i$ be the set of simplices in the triangulation $\cT_h$ which intersect $Q_i$;  see Figure~\ref{fg:cubes3}.  

Next a value $q > d$ is selected arbitrarily with the intention of applying Theorem~\ref{th:jamet} (since we are in the setting $k + 1 - m = 1$).  %A local estimate over $Q_i$ in the $W^{1,p}$-norm is translated to the $W^{1,q}$ norm, Theorem~\ref{th:jamet}is applied to each nearby simplex and then the estimate returned to the $W^{1,p}$-norm based on the mollification procedure.  
Applying Proposition~\ref{pr:lpq} followed by the definition of $\cK_i$ yields,
\begin{align*}
\wmpsn{\cM_h u - \Pi_{1} \cM_h u}{k}{p}{Q_i}^q & \apprle |Q_i|^{\frac{q-p}{p}}\wmpsn{\cM_h u - \Pi_{1} \cM_h u}{k}{q}{Q_i}^q\\
 & \apprle |Q_i|^{\frac{q-p}{p}} \sum_{K\in \cK_i} \wmpsn{\cM_h u - \Pi_{1} \cM_h u}{k}{q}{K}^q.
\end{align*}
Theorem~\ref{th:jamet} can now be applied to each term of the summation on the right hand side and Proposition~\ref{pr:lpqmol} is used to return the estimate to the correct norm:
\begin{align*}
 \wmpsn{\cM_h u - \Pi_{1} \cM_h u}{k}{p}{Q_i}^q & \apprle |Q_i|^{\frac{q-p}{p}}\sum_{K\in \cK_i} \left(\frac{h}{\left(\cos \theta_K\right)^k}\right)^q\wmpsn{\cM_h u}{k+1}{q}{K}^q\\
% & \apprle |Q_i|^{\frac{q-p}{p}} \left(\frac{h}{\left(\cos \theta\right)^k}\right)^q\wmpsn{\cM_h u}{k+1}{q}{R_i}^q\\
 & \apprle |Q_i|^{\frac{q-p}{p}} |h|^{d\frac{p-q}{p}} \left(\frac{h}{\left(\cos \theta\right)^k}\right)^q\wmpsn{u}{k+1}{p}{S_i}^q.
\end{align*}  
Since $|Q_i| = h^d$, we conclude that
\begin{equation}\label{eq:w1plocal}
\wmpsn{\cM_h u - \Pi_{1} \cM_h u}{k}{p}{Q_i} \apprle \frac{h}{\left(\cos \theta\right)^k}\wmpsn{u}{k+1}{p}{S_i}.
\end{equation}
Finally these estimates can be summed to get the global estimate:
\begin{align}
& \wmpsn{M_h u - \Pi_{1} \cM_h u}{k}{p}{\Omega}^p  \leq \sum_i \wmpsn{\cM_h u - \Pi_{1} \cM_h u}{k}{p}{Q_i}^p \notag \\
 & \hspace{.2in} \apprle \left(\frac{h}{\left(\cos \theta\right)^k}\right)^p \sum_i \wmpsn{u}{k+1}{p}{S_i}^p \apprle \left(\frac{h}{\left(\cos \theta\right)^k}\right)^p \wmpsn{u}{k+1}{p}{\Omega}^p.\label{eq:usesextension}
\end{align}
The final inequality holds since each cube $Q_i$ belongs to at most $5^d$ cubes $S_j$.
\end{proof}

\begin{figure}
\begin{center}
\psfrag{jh}{$h$}
\psfrag{jQ}{$Q_i$}
\psfrag{jR}{$R_i$}
\psfrag{jS}{$S_i$}
\includegraphics[width=.6\textwidth]{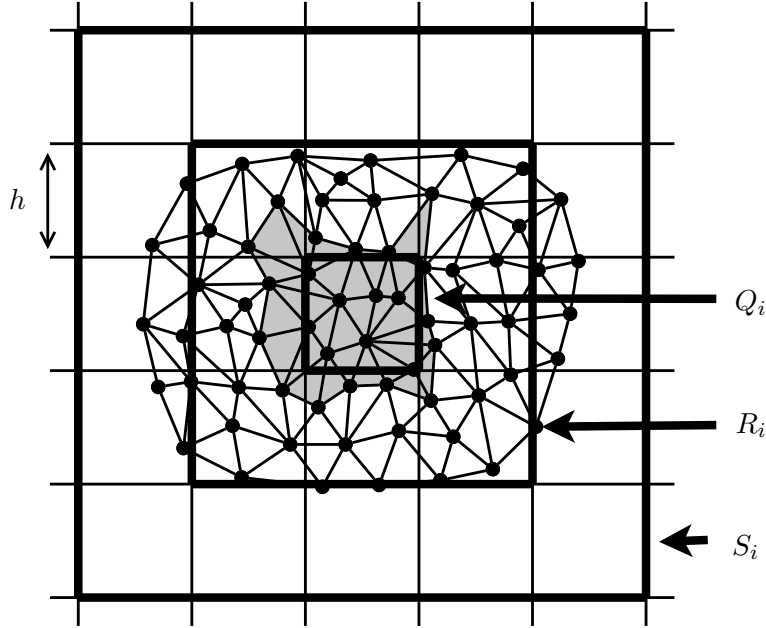}
\end{center}
\caption{Diagram for the proof of Theorem~\ref{th:w1uniform}.  Triangles belonging to $\cK_i$ are shaded. }\label{fg:cubes3}
\end{figure}

{\it Remark \arabic{section}.\addtocounter{remark}{1}\Alph{remark}}.
The essential difference between this analysis and that of Acosta~\cite{Ac01} is the use of the artificial (and bounded-ply) mesh $\cQ$ for establishing the local estimate (\ref{eq:w1plocal}).  

{\it Remark \arabic{section}.\addtocounter{remark}{1}\Alph{remark}}.
Looking closely at the inequality (\ref{eq:usesextension})  reveals that the extension domain property has been used to estimate $\wmpsn{u}{2}{p}{\bigcup S_i}$ by $\wmpsn{u}{2}{p}{\Omega}$.

{\it Remark \arabic{section}.\addtocounter{remark}{1}\Alph{remark}}. 
In the $L^p$ case of (\ref{eq:w2p}), i.e, $k=0$, the estimate does not depend on simplex shape which agrees with the classical result for more regular data.

%{\it Remark \arabic{section}.\addtocounter{remark}{1}\Alph{remark}}.
%{\bf FIX THIS REMARK!} The interpolant $\Xi_h u$ does not necessarily match the function's boundary values which is typically required by finite element spaces associated with Dirichlet problems.  Using a simple cutoff function Theorem~\ref{th:uniformlp} can be modified to preserve homogeneous boundary conditions~\cite{Ra09}.  However, this approach fails to produce a estimate on the $W^{1,p}$-norm. Constructing an average interpolant which preserves boundary values (in the spirit of \cite{SZ90}) and allows for high ply meshes remains an open problem.  

\section{Locally Quasi-Uniform Error Estimates}\label{sc:nonuniform}
\setcounter{remark}{0}
Extending Theorem~\ref{th:w1uniform} to admit non-uniform meshes requires a precise characterization of the mesh size $h$ at each point in the domain.  
We emphasize that the size function does not necessarily reflect the size of the triangles in the mesh.
Rather it is only an upper bound.
Because the average interpolant depends on a local neighborhood, the sizing function should not vary too rapidly so that there is roughly a unique size (up to a constant factor) associated with the neighborhood.

Before proving error estimates, we outline precise restrictions on these sizing functions.  Then based on the sizing function, variable-radius mollification and compatible cubic coverings are defined.  As before, the final interpolation operator is a composition of the (modified) mollification operation $\coM_h$ with the usual Lagrange interpolation operator $\Pi_1$:
\begin{equation}
\overline{\Xi}_{h}u(\bx) := \Pi_1 \coM_h u(\bx).
\end{equation}

\subsection{Sizing Function}\label{ss:sizingfun}
Using Lagrange interpolation for smooth (enough) functions, mesh size is captured locally by the length of the longest edge of a given triangle. 
For average interpolation, function values over neighboring triangles must be considered and thus the size of the averaging region also plays a role in the error estimate. 
On bounded aspect ratio meshes this is not an issue since the longest edges of any pair of neighboring triangles are within a constant factor~\cite{Mi94}.
To analyze average interpolants for meshes without aspect ratio bounds, we decouple the mesh sizing function from the actual size of the triangles.
Rather, the sizing function will be a smooth (enough) function which provides an upper bound on the size of nearby triangles in the mesh. 
The result will be a sizing function which can be used in the error estimate despite admitting meshes with adjacent triangles with dramatically different sizes.

Letting $\Csize$ be a positive constant, we consider the class of sizing functions $h:\R^d \rightarrow (0,\infty)$ with bounded first and second derivatives:

\noindent \begin{minipage}{.49\linewidth}
\begin{equation}\label{eq:gradh}
\vsn{\nabla h(\bx)} \leq \frac{1}{16\sqrt{d}};
\end{equation}
\end{minipage}
\begin{minipage}{.49\linewidth}
\begin{equation}\label{eq:varmolch}
\vsn{\partial_i\partial_j h(\bx)} \leq \Csize.
\end{equation}
\end{minipage}

\vspace{.1in}
\noindent We call a triangulation $\cT$ of $\Omega$ \emph{compatible with} sizing function $h$ if for all $\bx\in \Omega$, $h(\bx) \geq h_{T(\bx)}$ where $T(\bx)$ denotes the triangle containing $\bx$.%
\footnote{Without causing confusion this is a slight abuse of notation using $h(\bx)$ to denote the sizing function and $h_T$ to denote the length of the longest edge of triangle. This is done to emphasize the fact that the sizing function plays the role of the mesh size in the upcoming estimates.}
In this setting, we will establish uniform interpolation estimates for any sizing function with a compatible mesh as long as the sizing function satisfies (\ref{eq:gradh}) and (\ref{eq:varmolch}) and the mesh satisfies the maximum angle condition (in the sense of coplanarity).

{\it Remark \arabic{section}.\addtocounter{remark}{1}\Alph{remark}}.
The particular choice of constant in (\ref{eq:gradh}) is not essential.  However, this value is very convenient for localizing future estimates to one layer of neighboring cubes in a 2:1 balanced cubic partition described in Subsection~\ref{ss:cubiccovering}.  
Selecting a larger constant in (\ref{eq:gradh}) will require analysis of more levels of the cubic partition.

{\it Remark \arabic{section}.\addtocounter{remark}{1}\Alph{remark}}. Restriction (\ref{eq:gradh}) is implicitly present (although the constant is often weaker) in the context of shape regular meshes. For a shape-regular mesh $\cT$ and letting $P(T)$ be the union of triangles in $\cT$ which intersect triangle $T$, we define 
\[
h(\bv_i) = \max_{T\in \cT,\bv_i \in P(T)} h_T,
\]
This definition of $h$ can be extended to $\Omega$ using linear interpolation over the triangulation and it immediately follows that the triangulation $\cT$ is compatible with $h$.
Moreover, properties of well-spaced point sets ensure that $\vsn{\nabla h}$ can be bounded depending only on the mesh shape regularity~\cite{Mi94,Ta97,HMP06}. 
In other words, a shape regular mesh can be viewed as a mesh that is compatible with a Lipschitz sizing function, and thus (\ref{eq:varmolch}) is the only \emph{extra} condition on the sizing function we require to handle less shape-regular meshes.

{\it Remark \arabic{section}.\addtocounter{remark}{1}\Alph{remark}}.\label{rm:possiblesize}
Restrictions (\ref{eq:gradh}) and (\ref{eq:varmolch}) do allow for non-uniform meshes. 
The sequence of sizing functions $h_n(\bx) = \vsn{\bx}^2 + \frac{1}{2^n}$ can be paired with a sequence of meshes with progressively smaller elements near the origin. 
However, (\ref{eq:varmolch}) does pose some additional limitations on how \emph{quickly} the sizing function can approach zero: in the shape regular case common adaptive refinement matches a sizing function $h_n(\bx) = \vsn{\bx} + \frac{1}{2^n}$ or $h_n(\bx) = \min \left\{ \vsn{\bx},\frac{1}{2^n}\right\}$ which is not smooth enough for (\ref{eq:varmolch}). See Figure~\ref{fg:sizingexample}.

\begin{figure}
\centering
\begin{tabular}{c|c}
\includegraphics[height=.23\textwidth]{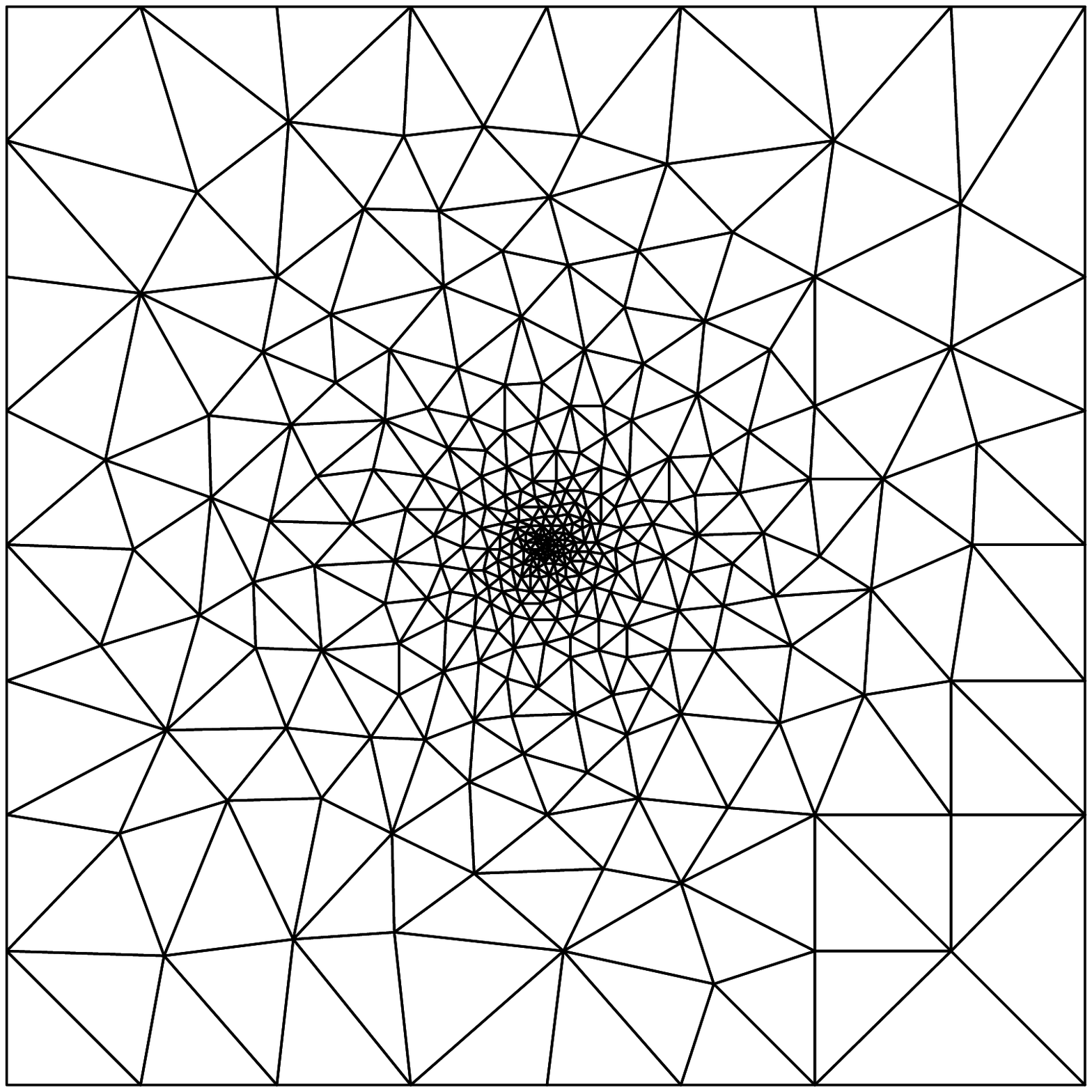}
\includegraphics[height=.22\textwidth]{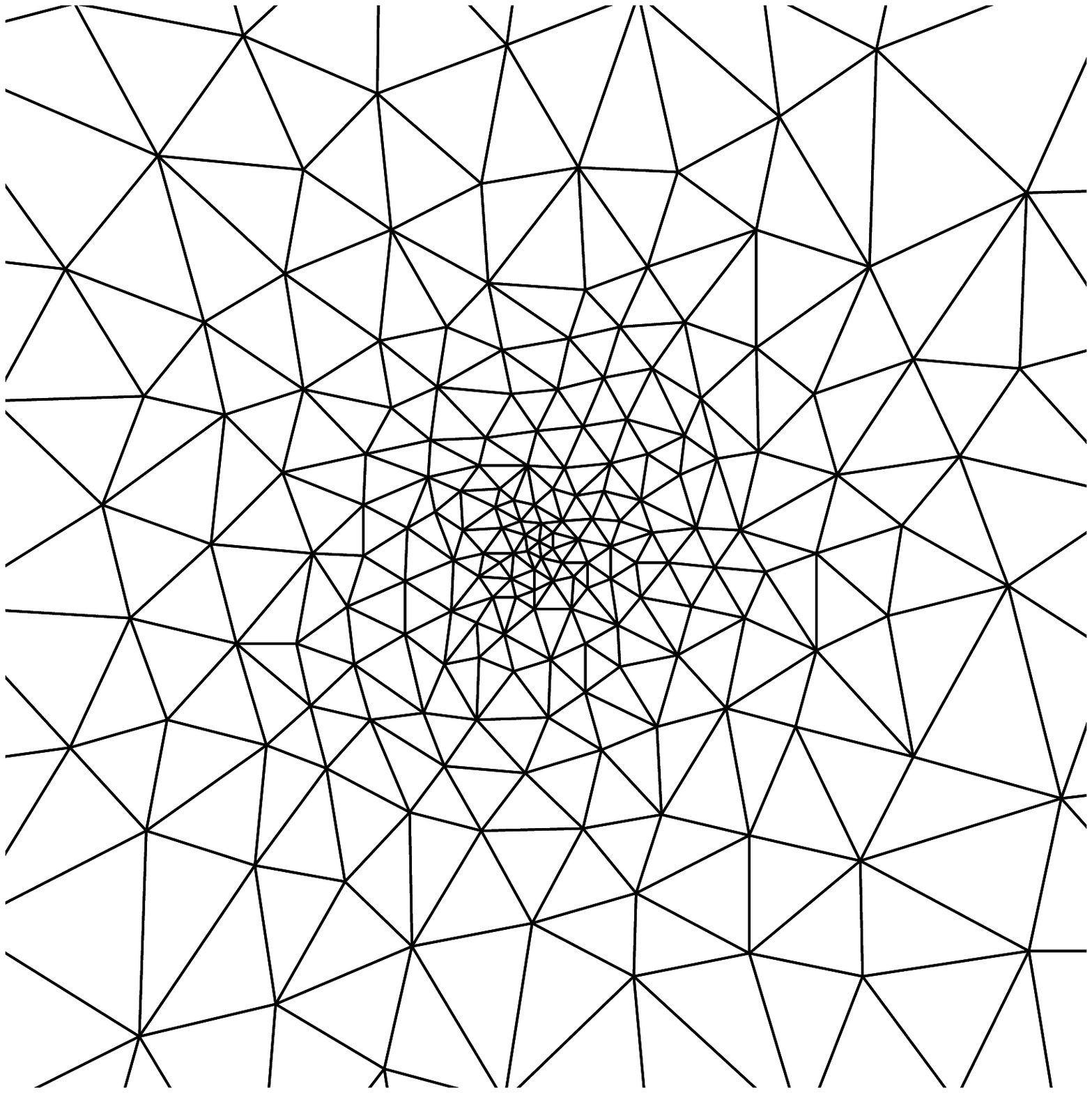} &
\includegraphics[height=.23\textwidth]{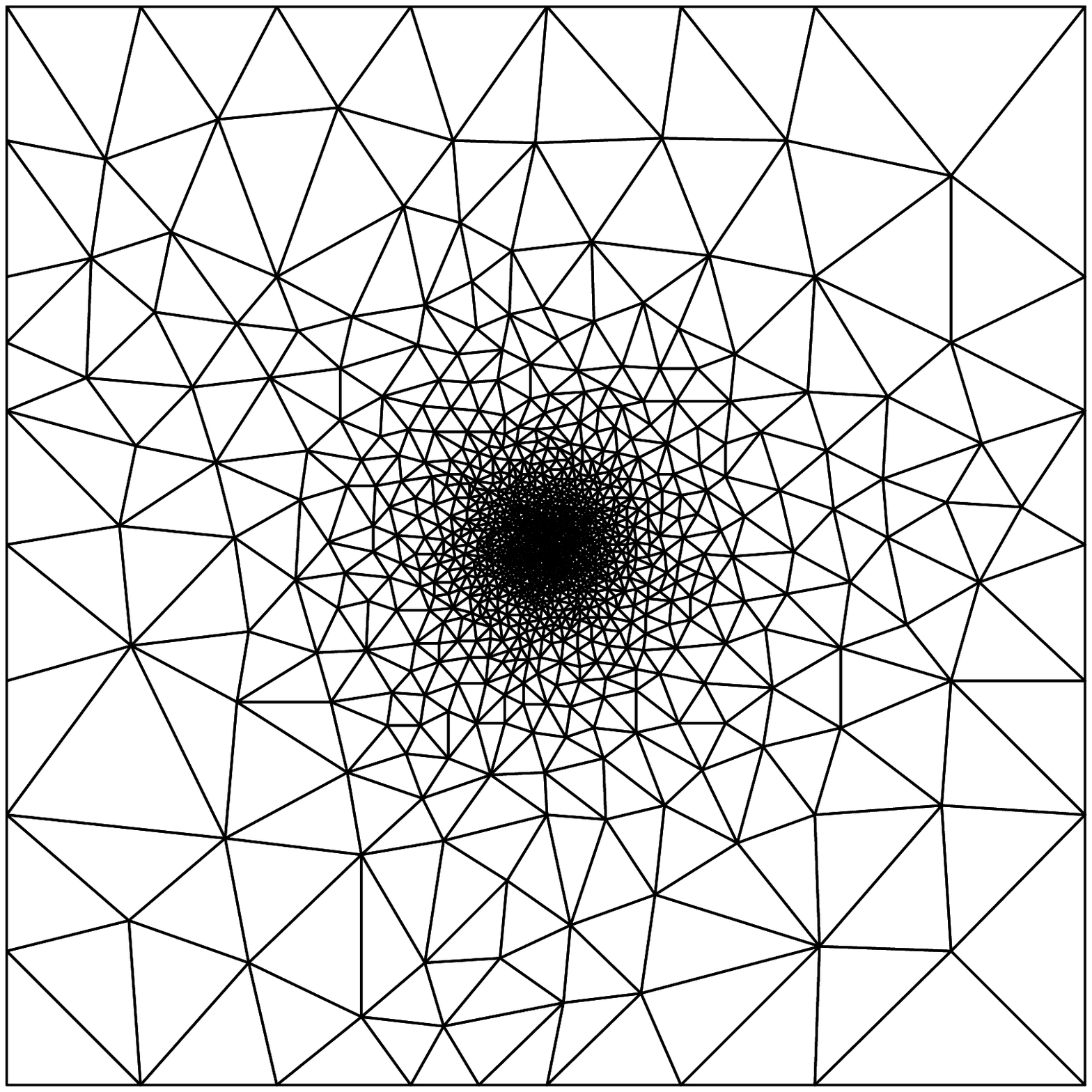}
\includegraphics[height=.22\textwidth]{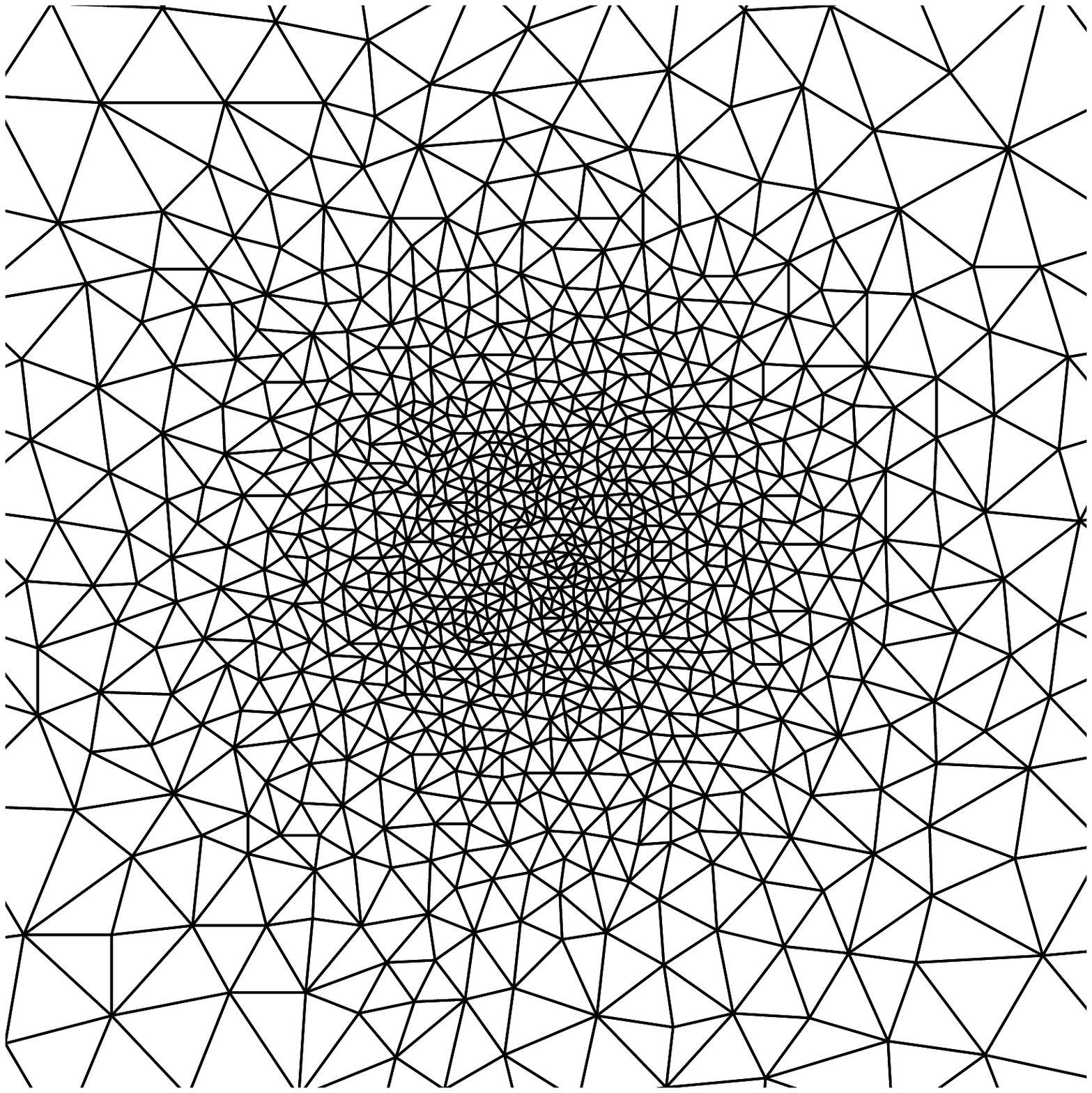}
\end{tabular}
\caption{Comparison of two (bounded aspect ratio) meshes associated with sizing functions $h(\bx) = \frac{1}{2}\vsn{\bx} + 0.01$ (left) and $h(\bx) = \vsn{\bx}^2 + 0.01$ (right) over the square $[-1,1]^2$. In each case, both the full mesh and the center region zoomed four times are shown. These meshes demonstrate the differences between conditions (\ref{eq:gradh}) and (\ref{eq:varmolch}): meshes only satisfying the former condition can grade more quickly away from zero. Note: the meshes do not satisfy (\ref{eq:gradh}) with the particular choice of constant, but this exaggerated example is given to more clearly display the effect.}\label{fg:sizingexample}
\end{figure}

{\it Remark \arabic{section}.\addtocounter{remark}{1}\Alph{remark}}.
Restrictions (\ref{eq:gradh}) and (\ref{eq:varmolch}) are somewhat weaker than those used in~\cite{Ac01} where there are restrictions essentially of the form $h(\bx) \leq \nabla h(\bx)$. 
In that setting Gronwall's inequality prevents such sizing functions from being truly non-uniform.
Using the techniques of this paper, it should be possible to recover Acosta's anisotropic results in a similar relaxed setting.

\subsection{Variable-Radius Mollification}\label{ss:variablemol}
We now consider the generalization of mollification by allowing the mollification radius to vary over the domain.  Given $u\in L^1(\R^d)$ and sizing function $h$, define 
\[
\coM_h u(\bx) := \int_{\R^d} u(\bx-\by) \rho_{h(\bx)}(\by) \d \by = \int_{\R^d} u(\bx+h(\bx)\by) \rho(\by) \d \by.
\]
%{\bf FIXME: compare Acosta to other references. $\int_{\R^d} u(\bx-\by) \rho_{h(\bx)}(\by) \d \by$}
The only difference between this definition and the standard mollifier is that the radius $h(\bx)$ is spatially varying. 
We consider an isotropic mollification region (and thus produce isotropic error estimates). 
The prior construction of Acosta uses an anisotropic mollifier, i.e., the term $\left(h_1(\bx) y_1,h_2(\bx)y_2,h_3(\bx)y_3\right)$ replaces $h(\bx)\by$ in the above definition~\cite{Ac01}. 
First we note a simple property of the variable radius mollifier.

\begin{proposition}\label{pr:varmolprop}
Let $p \geq 1$ and $Q\subset\R^d$, and define $Q_h:=\bigcup_{\bx \in Q} B(\bx,h(\bx))$. Then for all the sizing functions satisfying (\ref{eq:gradh}) and for all $u\in L^p(Q_h)$,
\begin{equation}\label{eq:varmolprop}
\int_Q |u(\bx+h(\bx)\by)|^p \d\bx \apprle \lpn{u}{p}{Q_h}^p.
\end{equation}
\end{proposition}
\begin{proof}
This results from a simple change of variables, $\bz=\bx+h(\bx)\by$. The Jacobian of this transformation can be bounded by $\frac{1}{(1+\nabla h(\bx)\cdot \by)d}$. By assumption (\ref{eq:gradh}) and since $\by \in B(\bzero,1)$, $1+\nabla h(\bx)\cdot\by \in [1-1/(16\sqrt{d}),1+1/(16\sqrt{d})]$ producing (\ref{eq:varmolprop}).
\end{proof}

Extensions of Propositions \ref{pr:dermol} and \ref{pr:lpqmol} are necessary to prove the non-uniform error estimates.  Since Remark~2.B %\ref{rm:commutemol} 
no longer applies, i.e., 
\[
\frac{\p }{\p x_i} \left(\coM_h u\right)\neq \coM_h\left( \frac{\p}{\p x_i} u \right),
\]
we prove these lemmas directly in the correct Sobolev norms. 

\begin{lemma}\label{lm:vardermol}
Let $p \geq 1$ and $Q$ be a cube.  Then for $k\in\{0,1\}$ and for all $u\in W^{k,p}(\Omega)$,
\begin{equation}\label{eq:lmvdm1}
\wmpsn{u - \coM_h u}{k}{p}{Q} \apprle H_Q \wmpn{u}{k+1}{p}{Q_h}
\end{equation}
where $H_Q = \max_{\bx \in Q} h(\bx)$, and $Q_h = \bigcup_{\bx \in Q} B(\bx,h(\bx))$.
\end{lemma}
\begin{proof}
The $k=0$ case of inequality (\ref{eq:lmvdm1}) follows from a very similar argument to the proof of Proposition~\ref{pr:dermol}.
Letting $w\in L^{p'}(\R^d)$ be an arbitrary function,
\begin{align*}
\int_Q |(u(\bx) - & \coM_h u(\bx) ) w(\bx) | \d \bx \\
 & = \int_Q \left| \left(\int_{B({\bf 0},1)}  \left( \int_0^1 \frac{\d}{\d s} (u(\bx - sh(\bx) \by)) \d s \right) \rho(\by) \d \by\right)  w(\bx) \right|\d \bx\\
 & \leq  \int_0^1 \int_{\R^d} \int_{Q}  \left|h(\bx) \by \cdot \nabla u(\bx - sh(\bx) \by) \rho(\by)   w(\bx) \right| \d \bx \d \by \d s.
\end{align*}
Next we note that $|h(\bx) \by| \leq H_Q$ and then bound the inner integral independent of $\by$ by expanding the set $Q$:
\begin{align*}
\int_Q |(u(\bx) - & \coM_h u(\bx) ) w(\bx) | \d \bx \\
 & \leq H_Q \int_0^1 \int_{\R^d} \rho(\by) \int_Q \vsn{\nabla u(\bx - sh(\bx)\by) w(\bx)} \d\bx \d\by \d s \\
 & \apprle H_Q \wmpsn{u}{1}{p}{Q_h} \lpn{w}{p'}{Q}.
\end{align*}
Note: the construction $\int \rho(\by)\d\by = 1$ was used to eliminate the integral in the $\by$ variable. 
%%%% COMMENT: This didn't make sense...
%and (\ref{eq:youngslp}) has been applied to produce the $L^p$ norm of $u$. 
Selecting $w=\left(u(\bx) - \coM_h u(\bx)\right)^{p-1}$ completes the result.  

Next, the $k=1$ case of inequality (\ref{eq:lmvdm1}) is addressed.   We begin by computing $\frac{\partial}{\partial x_i} \coM_h u(\bx)$ directly:
\begin{equation}\label{eq:dercomhu}
\frac{\partial}{\partial x_i} \coM_h u(\bx)  = \coM_h \partial_i u(\bx) 
+ \int_{\R^d}\rho(\by)\partial_i h(\bx) \nabla u(\bx+h(\bx)\by)\cdot \by \d\by.
\end{equation}
Then letting $g(\bx) := \int_{\R^d}\rho(\by)\partial_i h(\bx) \nabla u(\bx+h(\bx)\by)\cdot \by \d\by$,
\begin{equation}\label{eq:lmvdm}
\lpn{\frac{\partial}{\partial x_i}\left(u - \coM_h u\right)}{p}{Q} \leq \lpn{\partial_i u - \coM_h (\partial_i u)}{p}{Q} + \lpn{g}{p}{Q}.
\end{equation}

The first term can be estimated using the argument applied to (\ref{eq:lmvdm1}) yielding,
\begin{equation}\label{eq:lmvdmpart1}
\lpn{\partial_i u - \coM_h (\partial_i u)}{p}{Q} \leq H_Q \wmpsn{\partial_i u}{1}{p}{R}.
\end{equation}
It only remains to estimate $\lpn{g}{p}{Q}$. Let $\R^d_+$ denote the halfspace with positive first coordinate, $x_i > 0$. First $g$ is rewritten in terms of second derivatives of $u$:
\begin{align*}
g(\bx)  & = \partial_i h(\bx) \int_{\R^d}\rho(\by) \nabla u(\bx+h(\bx)\by)\cdot \by \d\by\\
 &= \partial_i h(\bx) \int_{\R^d_+}\rho(\by) \left(\nabla u(\bx+h(\bx)\by)-\nabla u(\bx-h(\bx)\by)\right)\cdot \by \d\by\\
 & = \partial_i h(\bx) \int_{\R^d_+}\rho(\by) \left(\int_{-1}^1\frac{\partial}{\partial s}\nabla u(\bx+sh(\bx)\by)\d s\right)\cdot \by \d\by.
\end{align*}
Our selection of a radially symmetric mollifier is important above as the property $\rho(\by) = \rho(-\by)$ has been applied. 
\begin{align*}
\int \vsn{ g(\bx) w(\bx)}\d\bx 
& \leq \int_Q \int_{\R^+} \int_{-1}^1  \vsn{\partial_i h(\bx)} \rho(\by) \vsn{ \frac{\partial}{\partial s}\nabla u(\bx+sh(\bx)\by)}\vsn{\by} \vsn{w(\bx)} \d s \d\by \d\bx\\
 & \apprle \int_{-1}^1 \vsn{s} \int_{\R^+} \rho(\by) \int_Q    \vsn{h(\bx)} \vsn{D^2u(\bx+sh(\bx)\by)}  \vsn{w(\bx)} \d \bx \d\by \d s,
\end{align*}
where $D^2u$ denotes the Hessian of $u$, the matrix of second derivatives. %{\bf Above, we used that $\vsn{y} < 1$ and ....} 
Next, H\"older's inequality can be applied:
\begin{align}\label{lm:lmvdmpart2}
\int \vsn{g(\bx)w(\bx) }\d\bx 
 & \apprle H_Q \wmpsn{u}{2}{p}{Q_h}\lpn{w}{p'}{Q}.
\end{align}
Finally setting $w(\bx)=g(\bx)^{p-1}$, and then combining (\ref{lm:lmvdmpart2}) with (\ref{eq:lmvdmpart1}) completes the result.
\end{proof}

%{\it Remark \arabic{section}.\addtocounter{remark}{1}\Alph{remark}}.
%Lemma~\ref{lm:vardermol} will be applied for cubes $Q$ such that $h_Q \approx H_Q \approx \size(Q)$.  In this case, the estimate gives the same factor of `$h$' seen in Proposition~\ref{pr:dermol}.  Construction of a suitable cubic spatial decomposition is performed in Subsection~\ref{ss:cubiccovering}.

The next step is to develop an extension of Proposition~\ref{pr:lpqmol} for the variable-radius mollifier.

%\frac{q-p}{pq}}
\begin{lemma}\label{lm:varmol}
If $1 \leq q \leq p \leq \infty$ and $Q$ is a bounded set, then for all the sizing functions satisfying (\ref{eq:gradh}) and (\ref{eq:varmolch}) and for all $u\in W^{2,q}(Q)$, $k\in\{0,1,2\}$,
\begin{equation}\label{eq:varmol}
\wmpsn{\coM_h u}{k}{p}{Q} \apprle h_Q^{-d} \vsn{Q_h}^{\frac{q-1}{q}}\vsn{Q}^{\frac{1}{p}} \wmpn{u}{k}{q}{Q_h}
\end{equation}
where $h_Q = \min_{\bx \in Q} h(\bx)$ and $Q_h := \bigcup_{\bx\in Q} B(\bx,h(\bx))$.  
\end{lemma}
\begin{proof}
We begin with the case $k=0$. For all $x\in Q$,
\begin{align*}
\vsn{\coM_h u(\bx)} & \leq \lpn{\rho}{\infty}{\R^d} \int_{B(\bzero,1)} \vsn{u(\bx+h(\bx)\by)} \d\by\\ 
 & \leq \lpn{\rho}{\infty}{\R^d} \int_{Q_h} h(\bx)^{-d} \vsn{u(\bz)} \d\bz.
\end{align*}
The previous inequality results from a change of variables $\bz=\bx+h(\bx)\by$. (Note, $\bx$ is fixed in this transformation.) H\"older's inequality gives,
\begin{align}\label{eq:varmolk0}
\vsn{\coM_h u(\bx)} % & \leq \lpn{\rho}{\infty}{\R^d} h_Q^{-d} \lpn{u}{q}{Q_h} \lpn{1}{q'}{Q_h} \\
 & = \lpn{\rho}{\infty}{\R^d} h_Q^{-d} |Q_h|^{1/q'} \lpn{u}{q}{Q_h}.
\end{align}
Integrating  (\ref{eq:varmolk0}) in $\bx$ over $Q$ produces (\ref{eq:varmol}) in the $k=0$ case.

Next, we turn to the case $k=2$; the $k=1$ case follows from very similar arguments.
\begin{align}
\frac{\partial^2}{\partial x_i\partial x_j}\coM_h u(\bx)  = \int_{\R^d} \rho(\by) \biggl[ \biggr.
& \partial_i\partial_j u(\bx+h(\bx)\by)\,\, + \biggr. \notag \\
 & \partial_ih(\bx)\nabla\left(\partial_j u(\bx+h(\bx)\by)\right)\cdot \by \,\, + \notag\\
 & \partial_j h(\bx)\nabla\left(\partial_i u(\bx+h(\bx)\by)\right)\cdot \by \,\, + \label{eq:2dercomhu}\\
 & \partial_i h(\bx)\partial_j h(\bx)  \by \cdot D^2 u\left(\bx+h(\bx)\by)\right) \by \,\, + \notag\\
 & \partial_i\partial_j h(\bx)\nabla\left(u(\bx+h(\bx)\by)\right)\cdot \by  \biggl. \biggr] \d \by \notag
\end{align}
The first term is identical to the $k=0$ case using the function $\partial_i\partial_j u$ in the place of $u$. In the second, third, and fourth terms, the same argument holds after we recall (\ref{eq:gradh}) and the fact that $\vsn{\by} \leq 1$. 
While the argument is the same, the final term yields a few slight differences: (\ref{eq:varmolch}) is required and the resulting estimate is in terms of $\wmpsn{u}{1}{p}{Q_h}$ rather than $\wmpsn{u}{2}{p}{Q_h}$ seen in all the other cases.
This causes the right-hand side of (\ref{eq:varmol}) to contain the full Sobolev norm, rather than the semi-norm.
\end{proof}

{\it Remark \arabic{section}.\addtocounter{remark}{1}\Alph{remark}}.\label{rm:onlyw1a} 
In Lemma~\ref{lm:varmol} (\ref{eq:varmolch}) is only necessary for the case $k=2$.

{\it Remark \arabic{section}.\addtocounter{remark}{1}\Alph{remark}}.
Lemmas~\ref{lm:vardermol} and \ref{lm:varmol} will be applied for cubes $Q$ such that $h_Q \approx H_Q \approx \size(Q) \approx \size(Q_h)$.  In this setting the estimates yield the same factors of `$h$' as seen in the analog with fixed mollification radius.

\subsection{Cubic Covering of $\Omega$}\label{ss:cubiccovering}

An important idea in Section~\ref{sc:uniform} was the use of a cubic grid of size $h$ to cover the domain.  In the non-uniform case we seek to cover $\Omega$ with cubes such that each cube has side length proportional to $h(\bx)$ for $\bx \in Q$; see Figure~\ref{fg:covermeshNonUniform} for an example.  
This is possible due to restriction (\ref{eq:gradh}) and can be achieved through the simple algorithm given in Algorithm~\ref{al:cubesize}.

\begin{algorithm}
\caption{Cover $\Omega$ by cubes proportional to $h$.}\label{al:cubesize}
\begin{algorithmic}
\REQUIRE $\cQ = \{Q\}$ where $Q$ is a cube such that $\size(Q_i) > 4 \max_{\bx\in \Omega} h(\bx)$
%\ENSURE $y = x^n$
\WHILE{$\exists Q_i\in \cQ $ such that $\size(Q_i) > 8 \max_{\bx\in Q_i} h(\bx)$}
\STATE Remove $Q_i$ from $\cQ$
\STATE Subdivide $Q_i$ into $2^d$ new cubes of size $\frac{\size(Q_i)}{2}$.
\STATE Insert all of the new cubes which intersect $\Omega$ into $\cQ$.
\ENDWHILE
\end{algorithmic}
\end{algorithm}

\begin{figure}
\begin{center}
\includegraphics[width=.35\textwidth]{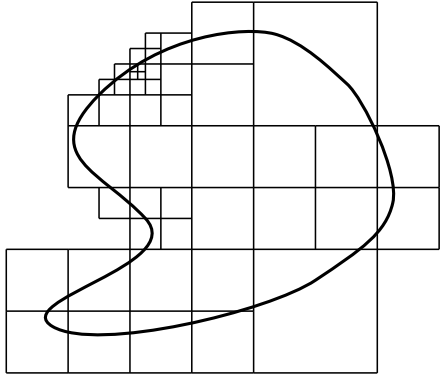}
\end{center}
\caption{An example cubic covering of a domain for some non-uniform sizing function.}\label{fg:covermeshNonUniform}
\end{figure}

First we assert that any neighboring cubes produced have proportional sizes.

\begin{proposition}\label{pr:cube21}
Let $\cQ = \{Q_i\}$ be the set of cubes resulting from Algorithm~\ref{al:cubesize}.  If $Q_i\cap Q_j \neq \emptyset$, then $\size(Q_i)\leq 2\,\size(Q_j)$.
\end{proposition}
\begin{proof}
Suppose $Q_i$ and $Q_j$ are adjacent cubes such that $\size(Q_i) \leq \frac{\size(Q_j)}{4}$.  Let $\bx_i \in Q_i \cap Q_j$ and let $\bx_j\in Q_j$.  Then,
\[
h(\bx_j) \leq h(\bx_i) + \frac{1}{16\sqrt{d}} |\bx_j - \bx_i| \leq \frac{\size(Q_i)}{4} + \frac{1}{16\sqrt{d}} \sqrt{d}\size(Q_j) \leq \frac{1}{8}\size(Q_j).
\]
Thus $Q_j$ is eligible to be split by Algorithm~\ref{al:cubesize} and we conclude that upon termination of the algorithm the lemma holds.
\end{proof}

The next proposition shows that for a compatible triangulation the simplices which intersect a particular cube are contained well inside the cube and its neighboring cubes.

\begin{proposition}\label{pr:cubeh}
Let $\cQ$ be the set of cubes resulting from Algorithm~\ref{al:cubesize}.  Then for all $\bx \in Q$, $4h(\bx) \leq \size(Q) \leq 16 h(\bx)$.
\end{proposition}
\begin{proof}
Let $Q$ be any cube resulting from Algorithm~\ref{al:cubesize} and let $\hat Q$ be the cube which was subdivided to form $Q$.  Then 
\begin{equation}\label{eq:cubeh1}
2\,\size(Q) = \size(\hat Q) > 8 \max_{\bx\in \hat Q} h(\bx) > 8 \max_{\bx\in Q} h(\bx).
\end{equation}
There exists $\bx\in Q$ such that $h(\bx) > \frac{\size(Q)}{8}$ since $Q$ is not split further by the algorithm.  By requirement (\ref{eq:gradh}), for any $\by\in Q$, 
\begin{equation}\label{eq:cubeh2}
h(\by) \geq h(\bx) - \frac{1}{16\sqrt{d}}|\bx-\by| \geq \frac{\size(Q)}{16}.
\end{equation}
Combining (\ref{eq:cubeh1}) and (\ref{eq:cubeh2}) yields the result.  
\end{proof}

\begin{lemma}\label{lm:cubesize}
Let triangulation $\cT$ be compatible with $h$ and 
let $\cQ = \{Q_i\}$ be the set of cubes resulting from Algorithm~\ref{al:cubesize}.  Then for all $Q \in \cQ$,
\[
\bigcup_{K\in \cK(Q)} \bigcup_{\bx\in K} B(\bx,h(\bx)) \subset S_i := \bigcup_{Q_j\cap Q \neq \emptyset} Q_j,
\]
where $\cK(Q)$ denotes the set of simplices intersecting $Q$.  
\end{lemma}
\begin{proof}
Let $K\in \cK(Q)$ and let $\bx \in K$.  Then $\bx$ belongs to $Q$ or a cube neighboring $Q$ since by Propositions \ref{pr:cube21} and  \ref{pr:cubeh} $\dist(\bx,Q) \leq h_K \leq \size(Q)/4 \leq \size(\hat Q)/2$ for any neighboring cube $\hat Q$.  Let $\hat Q$ denote the cube containing $\bx$ and let $\by \in B(\bx,h(\bx))$.  Also let $\bz$ denote the nearest point in $Q$ to $\bx$.  The result is shown in two cases depending on the size of $\hat Q$.  

\noindent \underline{Case 1}: $\size(\hat Q) = \size(Q)$ or $\size(\hat Q) = \size(Q)/2$.  Then 
\begin{align*}
\vsn{\by-\bz} & \leq \vsn{\by-\bx} + \vsn{\bx-\bz} \leq h(\bx) + h_K
  \leq \frac{\size(Q)}{4} + \frac{\size(Q)}{4} = \frac{\size(Q)}{2}.
\end{align*}
By the balance property (Proposition~\ref{pr:cube21}), this implies that $\by \in \bigcup_{Q_j\cap Q \neq \emptyset} Q_j$.  

\noindent \underline{Case 2}: $\size(\hat Q) = 2\, \size(Q)$.  Then 
\begin{align*}
\vsn{\by-\bz} & \leq \vsn{\by-\bx} + \vsn{\bx-\bz} \leq h(\bx) + h_K\\
 & \leq h(\bz) + \frac{\vsn{\bz-\bx}}{16\sqrt{d}} + \frac{\size(Q)}{4}
 \leq \size(Q) \left(\frac{1}{2} + \frac{1}{64\sqrt{d}}\right).
\end{align*}
The balance property (Proposition~\ref{pr:cube21}) ensures that $\by$ must belong to a cube of size greater than or equal to $\size(Q)$.  Since $\vsn{\by-\bz} \leq \size(Q)$, $\by \in \bigcup_{Q_j\cap Q \neq \emptyset} Q_j$.
\end{proof}

\subsection{Error Estimate}\label{ss:nonuniformee}

Now we have the necessary tools to prove an analog of Theorem~\ref{th:w1uniform} for a non-uniform sizing function and compatible triangulation.  

\begin{theorem}\label{th:nonuniformee}
Let $k\in\{0,1\}$. For sizing functions $h$ satisfying (\ref{eq:gradh}) and (\ref{eq:varmolch}) paired with any compatible mesh $\cT$, let $\cQ$ be the set of cubes produced by Algorithm~\ref{al:cubesize}.  Then for all $u\in W^{1,p}(\Omega)$,
\begin{equation}\label{eq:roughmolnon}
\wmpn{u - \overline{\Xi}_{h} u}{k}{p}{\Omega} \apprle \left(\sum_{Q_i\in\cQ} \left(\frac{h_i}{(\cos\theta_i^*)^k}\right)^p \wmpn{u}{k+1}{p}{Q_i}^p\right)^{\frac{1}{p}}
\end{equation}
where $h_i = \max_{\bx\in Q_i} h(x)$ and $\theta_i^* = \max_{T\cap S_i \neq \emptyset} \theta_T$.
\end{theorem}

Note that this result is uniform over the class of sizing functions satisfying (\ref{eq:gradh}) and (\ref{eq:varmolch}), i.e.,
 the inequality depends only upon $\Csize$ and not the specific sizing function or mesh selected.

\begin{proof}
The proof structure is very similar to Theorem~\ref{th:w1uniform}. First the estimate is divided using the triangle inequality and the intermediate function $\Xi_{h} u$.
\begin{equation}\label{eq:nonuniformproof}
\wmpsn{u - \Xi_{h} u}{k}{p}{Q_i} \leq \wmpsn{u - \coM_h u}{k}{p}{Q_i} + \wmpsn{\coM_h u - \Xi_{h} u}{k}{p}{Q_i}.
\end{equation}
Let $S_i$ denote the union of $Q_i$ with its neighboring cubes.  By Lemmas~\ref{lm:vardermol} and \ref{lm:cubesize}, 
\[
\wmpsn{u - \coM_h u}{k}{p}{Q_i} \leq h_i \wmpn{u}{k+1}{p}{S_i}.
\]
Next the second term in (\ref{eq:nonuniformproof}) is estimated. Select a value $q > d$. Proposition~\ref{pr:lpq} and Theorem~\ref{th:jamet} are applied for the smooth function $\coM_h u$:
\begin{align*}
\wmpsn{\coM_h u - \Pi_{1} \coM_h u}{k}{p}{Q_i}^q
 & \apprle \size(Q_i)^{\frac{q-p}{p}} \sum_{K\in \cK_i} \wmpsn{\coM_h u - \Pi_{1} \coM_h u}{k}{q}{K}^q\\
 & \apprle \size(Q_i)^{\frac{q-p}{p}}\sum_{K\in \cK_i} \left(\frac{h_K}{(\cos \theta_K)^k}\right)^q\wmpsn{\coM_h u}{k+1}{q}{K}^q.
\end{align*}
Lemma~\ref{lm:cubesize} is applied to combine terms in the summation and Lemma \ref{lm:varmol} is used to return the estimate to the correct norm. Letting $\theta_i = \max_{T\cap Q_i \neq \emptyset} \theta_T$, 
\begin{align}
 \wmpsn{\coM_h u - \Pi_{1} \coM_h u}{k}{p}{Q_i}^q  &\apprle \size(Q_i)^{\frac{q-p}{p}} \left(\frac{h_i}{(\cos \theta_i)^k}\right)^q \wmpsn{\coM_h u}{k+1}{q}{\cup_{K\in\cK_i} K}^q\notag \\
& \apprle  \size(Q_i)^{\frac{q-p}{p}} \left(\frac{h_i}{(\cos \theta_i)^k}\right)^q h_i^{\frac{p-q}{p}} \wmpn{u}{k+1}{p}{S_i}^q \notag\\
& \apprle  \left(\frac{h_i}{(\cos \theta_i)^k}\right)^q \wmpn{u}{k+1}{p}{S_i}^q. \label{eq:almostthere}
\end{align}  
Inequality (\ref{eq:almostthere}) follows from Proposition~\ref{pr:cubeh} which guarantees that $h_i^{-1} \size(Q_i) \apprle 1$.  
Finally summation over all cubes and again utilizing Proposition~\ref{pr:cubeh} finishes the result:
\begin{align*}
\sum_i  \wmpn{u - \overline{\Xi}_{h} u}{k}{p}{Q_i}^p & \apprle \sum_i \left(\frac{h_{Q_i}}{(\cos \theta_i)^k}\right)^p \wmpn{u}{k+1}{p}{S_i}^p\\
& \apprle \sum_i \left[\sum_{j {\rm \, s.t.\,} Q_j\cap Q_i \neq \emptyset} \left(\frac{h_{Q_i}}{(\cos \theta_i)^k}\right)^p\right] \wmpn{u}{1}{p}{Q_i}^p \\
&  \apprle 2^{(2d+1)p} \sum_i \left(\frac{h_{Q_i}}{(\cos \theta_i^*)^k}\right)^p \wmpn{u}{1}{p}{Q_i}^p.
\end{align*}
The final inequality results from the fact that there are fewer than $4^d$ cubes intersecting $Q_i$ and for any such cube $Q_j$, $h_{Q_j} \leq 2h_{Q_i}$.  
%The argument is completed with the same argument as Theorem~\ref{ss:nonuniformlp}: after summing over all cubes, we observe that each cube only appears in the right hand side at most $4^d$ times.  Rearranging the summations yields (\ref{eq:roughmolnon}).  
\end{proof}

{\it Remark \arabic{section}.\addtocounter{remark}{1}\Alph{remark}}.
Due to Remark~3.E, the $k=0$ case of Theorem~\ref{th:nonuniformee} does not depend upon (\ref{eq:varmolch}). This means that $L^p$ error estimates can be established for Lipschitz sizing functions (which, recalling Remark~3.C, result from adaptive refinement in many common settings).

{\it Remark \arabic{section}.\addtocounter{remark}{1}\Alph{remark}}.
Similarly to Theorem~\ref{th:w1uniform}, the extension domain property is also important to Theorem~\ref{th:nonuniformee}.  The usage is somewhat implicit: the continuous extension property is used to define $u$ on $Q_i \not\subset \Omega$.  

\section{Final Remarks}\label{sc:final}
\setcounter{remark}{0}

To extend the maximum angle condition and its generalizations to less regular data Section~\ref{sc:nonuniform} establishes ply-independent error estimates for a linear average interpolant. 
Broadly, this serves to demonstrate that requirements on shape regularity and data regularity can be weakened simultaneously without the inherent trade-off seen in previous analyses.  
Specifically, these results extend optimal error estimates under the maximum angle condition to several situations for which it was previously not shown: linear interpolation of functions in $W^{1,p}(\Omega)$ and (when $1\leq d/p$) $W^{2,p}(\Omega)$.  
The latter result is most important in three dimensions where the maximum angle condition does not hold for the Lagrange interpolant when $p\leq 2$.  

We have not addressed the question of boundary conditions. 
For error estimates in the $L^p$-norm, the function can be explicitly adjusted to match a homogeneous boundary condition. 
In \cite{CW08} averaging regions near the boundary are shifted outside the domain where an explicit extension by zero is used. 
When only considering the more straightforward scalar functions, an explicit cutoff can also accomplish the task~\cite{Ra09}. 
However, these manipulations corrupt error estimates of the derivatives and fully generalizing (\ref{eq:roughmolnon}) to preserve boundary data requires an alternative approach.
Ideally a mollification procedure analogous to the Scott and Zhang interpolant~\cite{SZ90} can be constructed which averages over a lower dimensional manifold so that only boundary data is taken into account.
Na\"ive approaches in this direction have proven unsuccessful.

\begin{figure}
\begin{tabular}{cc}
\parbox{.48\textwidth}{\includegraphics[width=.48\textwidth]{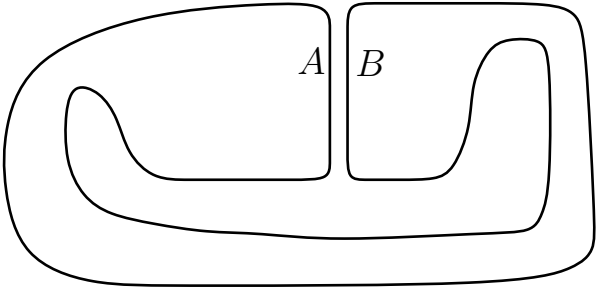}} &
\parbox{.48\textwidth}{\caption{Requiring a globally defined sizing function can lead to restrictions associated with non-local portions of the boundary; e.g., a small mesh in region $A$ will force a reasonably small mesh also in region $B$.}\label{fg:nonextension}}
\end{tabular}
\end{figure}

The extension domain property has been used somewhat implicitly throughout the arguments presented posing some limitations on the set of admissible meshes. 
Since the sizing function is defined globally, mesh size in non-local portions of the domain can impact the maximum mesh size; see Figure~\ref{fg:nonextension}. 
These limitations are often ignored in the analysis as asymptotically (as $h$ becomes small) the impact disappears since the domain is fixed. 
Other mollification-based average interpolants face similar limitations. 
For example, the construction of Christiansen and Winther explicitly construct a small extension region outside the domain boundary and mesh elements are assumed to be smaller than the size of this extension region.
Ideally the domain could be embedded in a manifold in which the non-local portions are far away (or more technically the extension operator has a much smaller constant).
While ad hoc methods for doing this are apparent for simple examples, a complete theory for general domains is not.

In some cases, a higher-order variant of Theorem~\ref{th:nonuniformee} may be necessary.  
Recall that for the expected convergence rate in the $W^{m,p}$-norm, Jamet's theorem requires $k + 1 - m > \frac{d}{p}$, where $k+1$ is the order of interpolation used and $d$ is the spatial dimension.  
Thus average interpolation is needed when $d$ is large, $p$ is near $1$, or $m$ is near $k$, i.e., estimates on higher derivatives are desired.  
Extending Theorem~\ref{th:nonuniformee} to cover these cases requires a higher-order mollification procedure.  While this mollifier can be defined, precise technical requirements under which the interpolation error estimate holds are still unknown.  

\section*{Acknowledgments} The author gratefully acknowledges Noel Walkington for many useful discussions
and the anonymous reviewers for several insightful comments.
%.  Discussion of Figure~\ref{fg:nonextension} is the result of some insightful comments of the anonymous reviewers.

%\newpage

%\bibliographystyle{abbrv}
\bibliographystyle{siam}
\bibliography{angles}

\appendix

\section{Coplanarity vs. the Maximum Angle Condition}\label{ap:copmac}

\begin{proof}(Proposition~\ref{pr:jametmac})
For this proof we omit the subscript $K$ on $\theta_K$ and $\psi_K$ without ambiguity.
First, we will bound $\psi$ from above by a function of $\theta$ to establish that if $\theta$ is bounded away from $\pi/2$, then $\psi$ is bounded away from $\pi$. 
This will be shown by direct construction: for a tetrahedron with maximum angle $\psi$, we will construct a particular vector $\hat{\bm\xi}$ which ensures that $\theta$ is sufficiently large. 
\begin{wrapfigure}{r}{.5\textwidth}
\psfrag{xihat}{$\hat{\bm\xi}$}
\psfrag{psi}{$\psi$}
\psfrag{bv1}{$\bv_1$}
\psfrag{bv2}{$\bv_2$}
\psfrag{bv3}{$\bv_3$}
\psfrag{bv4}{$\bv_4$}
\includegraphics[width=.5\textwidth]{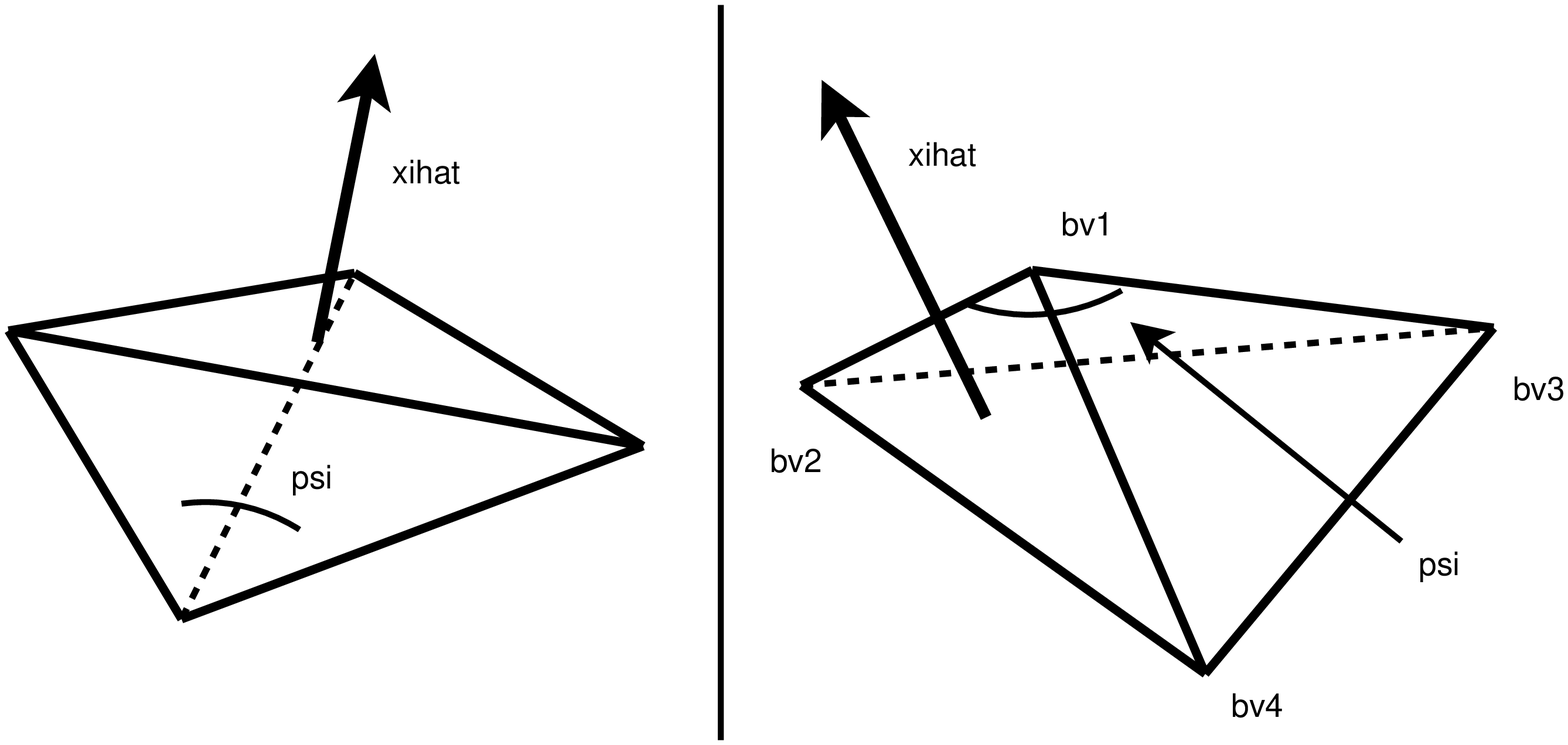}
\caption{In the proof of Proposition~\ref{pr:jametmac}, when bounding $\psi$ in terms of $\theta$, a specific $\hat{\bm\xi}$ can be constructed explicitly when $\psi$ is realized by a large dihedral angle (left) or a large angle in a face (right).}\label{fg:psitheta}
\end{wrapfigure}
Construction is divided into two cases. 
First we suppose that $\psi$ is realized by a dihedral angle of the tetrahedron.
Considering this largest dihedral angle, all vertices of the tetrahedron are contained in two planes that meet at angles of $\psi$ and $\pi-\psi$. 
Selecting a vector that points ``away'' from these planes demonstrates
\begin{equation}\label{thetapsi2}
\theta \geq \psi/2.
\end{equation}
Next we suppose that $\psi$ is realized by an angle in one of the faces of the tetrahedron. 
Let $\bv_1$ be the vertex at which the angle $\psi$ occurs, let $\bv_2$ and $\bv_3$ denote the other two vertices of that triangle and let $\bv_4$ denote the last vertex; see Fig.~\ref{fg:psitheta}. 
Moreover, without loss of generality we can assume that $|\bv_2-\bv_4| \leq |\bv_3-\bv_4|$, i.e., $\bv_4$ is nearer to $\bv_2$ than $\bv_3$.
Let $\hat{\bm\xi}$ be the unit normal vector to the triangle with vertices $\bv_1$, $\bv_2$ and $\bv_4$. 
To bound $\theta$, first observe that $\hat{\bm\xi}$ is orthogonal to three of the six lines containing edges of the tetrahedron. 
Two other edges belong to the triangle with vertices $\bv_1$, $\bv_2$, $\bv_3$, i.e., the triangle with a largest angle $\psi$. 
This large angle ensures that for any unit vector $\bu_i$ in the direction of either of these two lines, $\hat{\bm\xi} \cdot \bu_i \geq \alpha-\pi/2$. 
Lastly, we consider unit vectors in the direction of the line through $\bv_3$ and $\bv_4$. The triangle inequality implies
\begin{align*}
\left|\hat{\bm\xi} \cdot \left(\bv_3-\bv_4\right) \right|  &  \leq \left|\hat{\bm\xi} \cdot \left(\bv_3-\bv_2\right) \right| + \left|\hat{\bm\xi} \cdot \left( \bv_2- \bv_4\right) \right| = \left|\hat{\bm\xi} \cdot \left(\bv_3-\bv_2\right) \right|,
\end{align*}
where the second term in the sum is zero by construction of $\hat{\bm\xi}$ orthogonal to the a plane containing $\bv_2$ and $\bv_4$. 
Recalling our earlier argument, $\left|\hat{\bm\xi} \cdot \left(\bv_3-\bv_2\right) \right| \leq \left|\bv_3-\bv_2\right| \cos\left(\psi - \pi/2\right)$. Thus we can estimate the angle between $\hat{\bm\xi}$ and the line containing $\bv_3$ and $\bv_4$. 
\begin{align}\label{thetapsi3a}
\frac{\left|\hat{\bm\xi} \cdot \left(\bv_3-\bv_4\right) \right|}{\left|\bv_3-\bv_4 \right|} & \leq \frac{\left|\bv_3-\bv_2\right| \cos\left(\psi - \pi/2\right)}{\left|\bv_3-\bv_4 \right|} \leq 2 \cos\left(\psi - \pi/2\right).
\end{align}
The final inequality above follows from our definition of $\bv_3$: $\left|\bv_3-\bv_2\right| \leq \left|\bv_3-\bv_4\right| + \left|\bv_4-\bv_2\right| \leq 2 \left|\bv_3-\bv_4\right|$. 
We thus conclude that
\begin{align}\label{thetapsi3}
\theta \geq \min_i \hat{\bm\xi} \cdot \bu_i \geq \arccos \left( 2\cos\left(\psi - \pi/2\right)\right).
\end{align}
Considering the two expressions (\ref{thetapsi2}) and (\ref{thetapsi3}), we see that if $\theta$ is bounded away from $\pi/2$, then $\psi$ is bounded away from $\pi$.

Next we establish a reverse inequality. 
The crux of the argument is contained in an inequality from~\cite{AADL2011}.
Defining
%$m := \min\left\{\sin\frac{\pi-\psi}{2},\sin\psi \right\}$,
$m := \sin\frac{\pi-\psi}{2}$,
Acosta et al.~assert that there are unit vectors $\bu_1$, $\bu_2$, $\bu_3$ in the directions of the edges of the tetrahedron such that 
\begin{align}\label{eq:fromacosta}
\left|(\bu_1\times\bu_2)\cdot \bu_3\right| \geq m^3. \hspace{.3in}\textnormal{\cite[p. 147]{AADL2011}}
\end{align}
Letting ${\bm\xi}$ be an arbitrary unit vector, we consider the decomposition ${\bm\xi} = a (\bu_1\times\bu_2) + b \bu_{12}$ where $\bu_{12}$ is a unit vector in the plane spanned by $\bu_1$ and $\bu_2$. Since these are orthgonal subspaces, $a^2+b^2 = 1$ and thus $a \geq 1/\sqrt{2}$ or $b \geq 1/\sqrt{2}$. 
If $a \geq 1/\sqrt{2}$, then
\begin{align}\label{thetapsi7}
 \left|{\bm \xi} \cdot \bu_3\right| \geq a\left|(\bu_1\times\bu_2)\cdot \bu_3\right| \geq  \frac{m^3}{\sqrt{2}}.
\end{align}
Otherwise $b \geq 1/\sqrt{2}$. 
Letting $\gamma$ denote the \emph{non-acute} angle between the lines containing $\bu_1$ and $\bu_2$, we observe that (\ref{eq:fromacosta}) implies that $\sin\left(\pi-\gamma\right) \geq m^3$ or $\gamma \leq \pi - \arcsin(m^3)$. 
Also the angle between $\bu_{12}$ and the line in the direction of either $\bu_1$ or $\bu_2$ is no larger than $\gamma/2$.
 Thus,
\begin{align}\label{thetapsi8}
 \max_{i\in\{1,2\}} \left|{\bm \xi} \cdot \bu_i\right| \geq b\cos \left(\frac{\gamma}{2}\right) \geq \frac{1}{\sqrt{2}}\cos \left( \frac{\pi - \arcsin(m^3)}{2}\right). 
\end{align}
Combining (\ref{thetapsi7}) and (\ref{thetapsi8}), 
\begin{align*}
\theta = \max_{|{\bm\xi}|=1} \min_i \arccos {\bm\xi} \cdot \bu_i & \leq\arccos\left( \max \left\{\frac{m^3}{\sqrt{2}}, \frac{1}{\sqrt{2}}\cos \left( \frac{\pi - \arcsin(m^3)}{2}\right) \right\}\right). 
\end{align*}
Despite the complex form, this is the desired inequality.
Specifically, if $\psi$ is bounded away from $\pi$, then $m$ is bounded away from zero.
Since all the functions we consider are continuous, it follows that the argument of $\arccos$ is also bounded away from zero.
Thus $\theta$ is bounded away from $\pi/2$. 
\end{proof}

\section{Mollification}\label{ap:mol}
Proofs of several propositions in Section~\ref{ss:mollifiers} are provided here.

\begin{proof}(Proposition~\ref{pr:dermol})
Define $p'$ by $\frac{1}{p} + \frac{1}{p'} = 1$ and let $w\in L^{p'}(\R^d)$ be an arbitrary function.   
If $u\in C_c^\infty(\R^d)$ then
\begin{align*}
\irn |(u(\bx) - &\cM_h u(\bx) ) w(\bx) | \d \bx \\
 & = \irn \left|\left(u(\bx) - \irn u(\bx-\by) \rho_h(\by) \d \by\right)  w(\bx)\right| \d \bx\\
 & = \irn \left|\left(\irn  (u(\bx) -  u(\bx-\by)) \rho_h(\by) \d \by\right)  w(\bx) \right| \d \bx\\
 & = \irn \left| \left(\irn  \left( \int_0^1 \frac{\d}{\d s} u(\bx - s\by)) \d s \right) \rho_h(\by) \d \by\right)  w(\bx) \right|\d \bx.
\end{align*}
The final equality results from the fundamental theorem of calculus.  In the support of mollifier $\rho_h$, $|\by| \leq h$ so $|\by \rho_h(\by)|\leq h\rho_h(\by)$. Then H\"older's inequality can be applied:
\begin{align*}
\irn |(u(\bx) - &\cM_h u(\bx) ) w(\bx) | \d \bx \\
 & = \irn \left| \left(\irn  \left( \int_0^1 -\by \cdot \nabla u(\bx - s\by) \d s \right) \rho_h(\by) \d \by\right)  w(\bx) \right| \d \bx\\
 & \leq  \int_0^1 \irn \irn  \left|-\by \cdot \nabla u(\bx - s\by) \rho_h(\by)   w(\bx) \right| \d \bx \d \by \d s\\
 & \leq  \int_0^1 \irn  h \rho_h(\by) \left(\irn \vsn{\nabla u(\bx - s\by)}^p \d \bx \right)^\frac{1}{p}  \left(\irn  \vsn{w(\bx)}^{p'}\d \bx \right)^\frac{1}{p'} \d \by \d s\\
 & = h \vn{\nabla u}_{L^p(\R^d)}\vn{w}_{L^{p'}(\R^d)}.
\end{align*}
By selecting $w = \left(u - \cM_h u\right)^{p-1}$, the left-hand side of the above estimate becomes $\lpn{u-\cM_h u}{p}{\R^d}^p$.  Since $p'(p-1) = p$, $\vn{w}_{L^{p'}(\R^d)} = \vn{u-\cM_h u}_{L^p(\R^d)}^{p-1}$.  Dividing both sides by $\vn{w}_{L^{p'}(\R^d)}$ gives the desired inequality.  Density of smooth functions in $W^{1,p}(\Omega)$ is used to extend this result to the full space.  
\end{proof}

%\begin{proof}(Proposition~\ref{pr:dermol2})
%Define $p'$ such that $\frac{1}{p} + \frac{1}{p'} = 1$ and let $w\in L^{p'}(\R^d)$ be an arbitrary function. Then
%\begin{align*}
%\irn &\left| \frac{\p}{\p x_i}  \left(\cM_h u(\bx)\right) w(\bx)\right| \d \bx  = \irn \left| \left( \irn  u(\bx-\by) \frac{\p}{\p y_i} \rho_h(\by)  \d \by \right) w(\bx)\right| \d \bx\\
%& \leq \irn \irn \left| u(\bx-\by) \frac{\p}{\p y_i} \rho_h(\by) w(\bx) \right| \d \bx \d \by\\
%& \leq \irn \left[ \left| \frac{\p}{\p x_i} \rho_h(\by)\right| \left(\irn \left|u(\bx-\by)\right|^p \d \bx\right)^{\frac{1}{p}} \left(\irn \left|w(\bx)\right|^{p'} \d \bx\right)^{\frac{1}{p'}}  \right] \d \by\\
%& = \frac{1}{h}\lpn{u}{p}{\R^d}\lpn{w}{p'}{\R^d} \lpn{\frac{\p}{\p x_i} \rho}{1}{\R^d}.
%\end{align*}
%Selecting $w = \left( \frac{\p}{\p x_i}\left( \cM_h u\right) \right)^{p-1}$ and then summing over $i$ yields the result.  
%\end{proof}

\begin{proof}(Proposition~\ref{pr:lpqmol})
Let $\hat u(x) := u(x) \chi_{Q_h}(x)$ where $\chi_{Q_h}$ is the characteristic function of $Q_h$.  Applying Young's inequality to $\cM_h \hat u$ gives
\[\lpn{\cM_h \hat u}{p}{\R^d} \apprle \lpn{\rho_h}{\frac{pq}{pq+q-p}}{\R^d} \lpn{\hat u}{q}{\R^d}.\]
Since $\cM_h \hat u \equiv \cM_h u$ on $Q$, $\lpn{\cM_h u}{p}{Q} = \lpn{\cM_h \hat u}{p}{Q} \leq \lpn{\cM_h \hat u}{p}{\R^d}$.  Then recognizing that $\lpn{\hat u}{q}{\R^d} = \lpn{\hat u}{q}{Q_h}$ and computing the norm of $\rho_h$ yields the desired estimate.
\end{proof}

\end{document}

%% file: arandmacros.tex
\newcommand\R{\mathbb{R}}

\def \Csize{C_{\rm sizing}}
\def \bzero{{\bf 0}}
\def \bx{{\bf x}}
\def \by{{\bf y}}
\def \bz{{\bf z}}
\def \bu{{\bf u}}
\def \bv{{\bf v}}

\def \ba{{\bf a}}

\def \cM{{\cal M}}
\def \coM{\overline{\cal M}}
\def \cT{{\cal T}}
\def \cK{{\cal K}}
\def \cQ{{\cal Q}}

\def \d{\,{\rm d}}
\def \p{\partial}
\def \irn{\int_{\R^d}}

\newcommand\size{\textnormal{size}}

\newcommand\diam{\,\textnormal{diam}}
\newcommand\dist{\,\textnormal{dist}}

\newcommand{\vn}[1]{\left|\left|#1\right|\right|}
\newcommand{\vsn}[1]{\left|#1\right|}

\newcommand{\lpn}[3]{\vn{#1}_{L^{#2}(#3)}}

\newcommand{\wmpn}[4]{\vn{#1}_{W^{#2,#3}(#4)}}
\newcommand{\wmpsn}[4]{\vsn{#1}_{W^{#2,#3}(#4)}}

\newcommand{\comment}[1]{}